\documentclass[11pt]{amsart}
\usepackage{amsthm,amssymb,url,hyperref}
\usepackage{fullpage}
\usepackage{stmaryrd}
\usepackage[all]{xy}
\usepackage{color,tikz}

\theoremstyle{plain}
\newtheorem{theorem}{Theorem}[section]

\newtheorem{corollary}[theorem]{Corollary}

\newtheorem{lemma}[theorem]{Lemma}

\newtheorem{proposition}[theorem]{Proposition}

\theoremstyle{definition}

\newtheorem{example}[theorem]{Example}
\newtheorem{remark}[theorem]{Remark}

\def\ld{\backslash}

\def\ker#1{\mathrm{ker}(#1)}
\def\Ker#1{\mathrm{Ker}(#1)}
\def\aut#1{\mathrm{Aut}(#1)}
\def\Aut#1{\mathrm{Aut}(#1)}
\def\End#1{\mathrm{End}(#1)}
\def\aff#1{\mathrm{Aff}#1}
\def\Aff#1{\mathrm{Aff}#1}
\def\lmlt{\mathrm{LMlt}}
\def\dis{\mathrm{Dis}}

\def\Sym{\mathrm{Sym}}

\def\Con{\mathrm{Con}}

\def\eps{\varepsilon}

\def\setof#1#2{\{#1\, : \,#2\}}

\def\Z{\mathbb Z}

\def\Q{\mathcal Q}
\def\TC{\mathrm{TC}}
\def\O{\mathcal O}

\def\N{\mathrm{Norm}}
\def\con#1{\mathrm{con}_{#1}}

\numberwithin{equation}{section}

\def\cg#1{\equiv_\alpha}
\newcommand*\xbar[1]{%
   \hbox{%
     \vbox{%
       \hrule height 0.5pt 
       \kern0.5ex
       \hbox{%
         \kern-0.1em
         \ensuremath{#1}%
         \kern-0.1em
       }%
     }%
   }%
} 

\title{Commutator theory for racks and quandles}

\author{Marco Bonatto} 
\address[Bonatto]{IMAS-CONICET, University of Buenos Aires, Argentina}
\email[Stanovsk\'y]{marco.bonatto.87@gmail.com}

\author{David Stanovsk\'y}

\address[Stanovsk\'y]{Department of Algebra, Faculty of Mathematics and Physics, Charles University, Prague, Czech Republic}
\email[Stanovsk\'y]{stanovsk@karlin.mff.cuni.cz}

\begin{document}

\thanks{Research partly supported by the GA\v CR grant 18-20123S}


\date{\today}

\begin{abstract}
We adapt the commutator theory of universal algebra to the particular setting of racks and quandles, exploiting a Galois connection between congruences and
certain normal subgroups of the displacement group. Congruence properties, such as abelianness and centrality, are reflected by the corresponding relative displacement groups, and the global properties, solvability and nilpotence, are reflected by the properties of the whole displacement group. To show the new tool in action, we present three applications: non-existence theorems for quandles (no connected involutory quandles of order $2^k$, no latin quandles of order $\equiv2\pmod4$), a non-colorability theorem (knots with trivial Alexander polynomial are not colorable by solvable quandles; in particular, by finite latin quandles), and a strengthening of Glauberman's results on Bruck loops of odd order.
\end{abstract}

\maketitle

\section{Introduction}\label{sec:intro}

\subsection{Motivation}

The primary motivation for the development of the theory of racks and quandles comes from constructions of knot invariants \cite{CESY,EN,J}, describing set-theoretic solutions to the quantum Yang-Baxter equation \cite{EGS,ESS}, constructions of Hopf algebras \cite{AG}, or the abstract theory of quasigroups and loops \cite{Sta-latin}. In the present paper, we develop the concepts of \emph{abelianness} and \emph{centrality}, and the derived concepts of \emph{solvability} and \emph{nilpotence}, for racks and quandles, by adaptation of the general commutator theory of universal algebra \cite{FM} to the particular setting of racks. We aim at new tools to be used in a deeper study of rack theory and its applications.

The \emph{commutator theory}, as developed in universal algebra, originated in the 1970's works of Smith \cite{Smith-mv}, 
and culminated in the Freese-McKenzie monograph \cite{FM}, expanding the scope from Mal'tsev varieties to congruence modular varieties and beyond. The initial ideas developed into a rather deep theory that proved immensely useful in solving various problems of universal algebra and combinatorics of functions, see \cite{MS} for references. 
We also refer to \cite{SV1,SV2} for a successful adaptation of the commutator theory to quasigroups and loops (``non-associative groups"), which was an inspiration for the present work.

Quandles do not form a congruence modular variety, hence one cannot expect the best behavior of the congruence commutator. Nevertheless, the derived notions of abelianness and centrality, and subsequently solvability and nilpotence, are important structural concepts in any algebraic structure \cite{TCT}. The main idea behind our paper is that, in racks, they are well reflected inside the \emph{displacement group} of rack, thus allowing to use group-theoretical arguments to prove theorems about racks. The main results are stated in Section \ref{s:main_results}, and their proof occupies a major part of the paper. Several simple applications of the new tool are presented in Section 8, and more
involved results will be the subject of subsequent papers \cite{BB,BP,LSS,Nagy}.

The strong interplay between congruences of a rack and normal subgroups of its displacement group has been noticed already at the very dawn of quandle theory by Nagao \cite{Nag}, and later rediscovered and developed in various forms, for example, \cite{BLRY,Joyce-simple}. They all inspired our approach which is based on the Galois connection given by the operators $\dis_\alpha$ and $\con N$, developed in Section \ref{s:3}.

Our results are also relevant in the context of the general theory of quasigroups and loops. Latin quandles are also known as \emph{left distributive quasigroups} and have been studied extensively even before the quandle theory was born, see \cite{Sta-latin} for a survey. 
The concept of solvability was introduced on several independent occasions. Formal definitions were somewhat different, but all shared the property that solvable latin quandles had solvable left multiplication groups. Here we show that the latter property is equivalent to solvability in the sense of universal algebra, and using Stein's theorem \cite{SteA} we conclude that all finite latin quandles are solvable. Principal isotopy translates the results into the setting of loops. Using the Belousov-Onoi correspondence \cite{BO}, we obtain that Bruck loops of odd order are solvable in the sense of universal algebra, thus strengthening Glauberman's theorem \cite{G2} whose English statement was identical, however, using a weaker definition of solvability (see \cite{SV1} for a discussion). 
This is one of the rare examples when a result about loops is derived from the properties of their quasigroup isotopes, turning the usual flow of ideas. 

Our results are also relevant in the context of the theory of \emph{supernilpotence} \cite{AM}.
This is a stronger property than nilpotence, based on Bulatov's commutators of higher arity. A theorem by Kearnes \cite[Theorem 3.14]{Kea} says that a finite algebraic structure with a Mal'tsev term (a quasigroup, or a latin quandle, in particular) is supernilpotent if and only if it is a direct product of nilpotent algebras of prime power size. In this context, our Theorem \ref{t:nilpotent_prime_decomposition} and Corollary \ref{c:bruck} state that, for finite latin quandles and for Bruck loops of odd order, nilpotence is equivalent to supernilpotence.

Some of the properties of the connection between congruences and normal subgroups hold more generally, in any binary algebraic structure with bijective left translations, so called \emph{left quasigroup}. Our motivation for general formulation comes from the quantum Yang-Baxter equation: the non-degenerate set-theoretic solutions can be interpreted as a pair of left quasigroups, and in some cases, as a single left quasigroup. One example is the class of racks, and another one is the class of involutive solutions in Rump's notation \cite{Rump}. Some of our concepts are immediately useful in this context.

\subsection{Main results}\label{s:main_results}

Let $Q$ be a rack. We define the \emph{displacement group} of $Q$ by $\dis(Q)=\langle L_a L_b^{-1}: \, a,b\in Q\rangle$ where $L_u(x)=u*x$ is the left translation by $u$. For a congruence $\alpha$ of $Q$, we define \emph{the displacement group relative to $\alpha$} by $\dis_\alpha=\langle L_a L_b^{-1}: \, a\,\alpha\,b\rangle$. See Sections \ref{s:2} and \ref{s:3} for details. 

Let $\alpha$ be an equivalence on a set $X$ and let a group $G$ act on $X$. We call the action \emph{$\alpha$-semiregular} if for every $g\in G$, if $g(a)=a$ then $g(b)=b$ for every $b\,\alpha\,a$. 

The first main theorem, proved in Section \ref{s:commutators}, characterizes abelian and central congruences by group-theoretic properties of the corresponding relative displacement groups.
The proof actually works for a larger class, properly containing all racks, so called \emph{ltt left quasigroups}. The ltt property is a syntactic condition that requires all terms to take a particular equivalent form, explained in Section \ref{s:ltt}.

\begin{theorem}\label{t:abelian,central}
Let $Q$ be an ltt left quasigroup and $\alpha$ its congruence. Then
\begin{enumerate}
	\item $\alpha$ is abelian if and only if the subgroup $\dis_\alpha$ is abelian and it acts $\alpha$-semiregularly on $Q$;
	\item $\alpha$ is central if and only if the subgroup $\dis_\alpha$ is central in $\dis(Q)$ and $\dis(Q)$ acts $\alpha$-semiregularly on $Q$.
\end{enumerate}
\end{theorem}

As a special case, we obtain that an ltt left quasigroup $Q$ is abelian if and only if $\dis(Q)$ is abelian and acts semiregularly on $Q$. This partly extends the main result of \cite{JPSZ2}.

The second main theorem, proved in Section \ref{s:nilpotent,solvable}, characterizes solvability and nilpotence of a rack by the corresponding group-theoretic property of the displacement group. 

\begin{theorem}\label{t:nilpotent,solvable}
Let $Q$ be rack. Then
\begin{enumerate}
	\item $Q$ is solvable if and only if $\dis(Q)$ is solvable;
	\item $Q$ is nilpotent if and only if $\dis(Q)$ is nilpotent.
\end{enumerate}
\end{theorem}

\begin{corollary}\label{Latin are solv}
Finite latin quandles are solvable.
\end{corollary}

\begin{proof}
Finite latin quandles have solvable displacement groups by \cite[Theorem 1.4]{SteA}, hence Theorem \ref{t:nilpotent,solvable} applies.
\end{proof}

In Section \ref{s:nilpotent_prime_decomposition}, we prove the prime decomposition theorem for finite nilpotent quandles satisfying certain homogenity assumptions. In particular, it applies to latin quandles.

\begin{theorem}\label{t:nilpotent_prime_decomposition}
Let $Q$ be a finite connected faithful quandle. Then $Q$ is nilpotent if and only if $Q$ is a direct product of connected quandles of prime power size.
\end{theorem}

In Section \ref{s:7}, we present the construction of abelian and central extensions, inspired by \cite[Section 2.3]{AG} and \cite[Section 7]{FM}, and ask which surjective homomorphisms can be represented by these extensions. We show a positive result for central extensions (Proposition \ref{p:central_ext_rep}) and a negative result for abelian extensions (Example \ref{e:abelian_ext_rep} which stands in contrast to our postive results on abelian extensions of loops \cite[Theorem 4.1]{SV2}).

In the last section, we present three applications. The proofs are simple, yet, with previous methods, the results were either very complicated to prove, or inaccesible. 
First, we show two non-existence results: there are no connected involutory quandles of order $2^k$ (Theorem \ref{t:2^k}), and there are no latin quandles of order $\equiv2\pmod4$ (Theorem \ref{t:stein}). The latter is known as Stein's theorem \cite{SteS} and originally required a rather involved topological argument.
Next, we prove that knots and links with trivial Alexander polynomial are not colorable by any solvable quandle (Theorem \ref{t:coloring}), extending an analogical result for affine quandles \cite{Bae}.
Finally, we explain the Belousov-Onoi correspondence that translates our results for latin quandles into their loop isotopes.

\subsection{The structure of the paper}

In Section \ref{s:2}, we summarize the basic facts and observations about quandles, racks and left quasigroups, necessary to understand the rest of the paper. 
Section \ref{s:3} explains the Galois correpsondence between congruences and subgroups of the displacement group. In the end, we also discuss when the two operators $\dis_\alpha$ and $\con{N}$ give mutually inverse lattice isomorphisms.
In Section \ref{s:4}, we give a brief introduction to the commutator theory and present several basic facts about terms in left quasigroups, including the ltt property.
In Section \ref{s:5}, we adapt the commutator theory to ltt quasigroups and prove Theorem \ref{t:abelian,central}. Then we show that for faithful quandles, one can drop the semiregularity conditions, and, in turn, the commutator has better properties. We also calculate the center of a rack, and prove that medial racks are nilpotent.
In Section \ref{s:6}, we investigate nilpotence and solvability, and prove Theorems \ref{t:nilpotent,solvable} and \ref{t:nilpotent_prime_decomposition}.
Section \ref{s:7} is about abelian and central extensions, and Section \ref{s:8} contains the applications.

\section{Rack and quandle theoretic concepts} \label{s:2}

\subsection{Division in binary algebraic structures}

By an \emph{algebraic structure} we mean a non-empty set equipped with a collection of operations (of arbitrary finite arity). We will mostly consider algebraic structures with two binary operations.

Let $*$ be a binary operation on $Q$. For $a\in Q$, let
\begin{displaymath}
L_a:Q\to Q,\quad b\mapsto a\ast b;\qquad   R_a:Q\to Q,\quad a\mapsto b\ast a
\end{displaymath}
be the \emph{left translation} by $a$ and the \emph{right translation} by $a$, respectively.
If all left translations are bijective, we can define the left division operation by
\[   a\backslash b = L_a^{-1}(b).\]
The resulting algebraic structure $Q=(Q,*,\ld)$ will be called a \emph{left quasigroup}. Left quasigroups can be axiomatized by the identities $x\backslash(x*y)=y=x*(x\backslash y)$.
A left quasigroup is called \emph{involutory} if $L_a^2=1$ for every $a$, i.e., if the identity $x*(x*y)=y$ holds, or equivalently, if $*=\backslash$. 

If all left and right translations are bijective, we use the term \emph{quasigroup}, or we use the adjective \emph{latin}. The right division operation is defined analogically.

Many universal algebraic concepts, such as subalgebras, congruences and their properties (in particular, the centralizing relation $C(\alpha,\beta;\delta)$ that defines the commutator), are sensitive to the choice of operations. In our paper, left quasigroups, including racks and quandles, will always be considered as structures $(Q,*,\ld)$, including left division (and excluding right division in the latin case). In particular, substructures and quotients of left quasigroups are always left quasigroups.
 
For latin quandles, there is a collision with the standard setting of quasigroup theory, where both division operations are considered. We avoid this collision by stating all results only for finite quasigroups where the choice of operations is irrelevant, since both divisions can be defined by a multiplicative term: indeed, if $n$ is the least common multiple of orders of all left translations, then $a\ld b=L_a^{n-1}(b)=a*(a*(\ldots*(a*b)))$, and similarly for right division.

For a left quasigroup $Q$, we define two important subgroups of the symmetric group over the set $Q$: the \emph{left multiplication group} and the \emph{displacement group}
\[ \lmlt(Q)=\langle L_a: \ a\in Q\rangle,\qquad \dis(Q)=\langle L_a L_b^{-1}: \ a,b\in Q\rangle. \]
The group $\lmlt(Q)$ acts naturally on the set $Q$. Whenever we say that (a subgroup of) $\lmlt(Q)$ acts in some way, we implicitly mean the natural action on $Q$.

A left quasigroup $Q$ is called \emph{connected} if $\lmlt(Q)$ acts transitively on $Q$. Quasigroups are always connected, since $L_{b/a}(a)=b$ for every $a,b$.

\subsection{Racks and quandles}

A \emph{rack} is a left quasigroup in which all left translations are automorphisms. This can be expressed as an identity,
\begin{equation}
x*(y*z)=(x*y)*(x*z)\label{LD},\end{equation} called \emph{left self-distributivity}. 
An idempotent rack (i.e., where $x*x=x$ holds) is called a \emph{quandle}.
We refer to \cite[Sections 1-8]{J} or \cite[Section 2]{HSV} for a collection of basic properties of racks and quandles to be used in the present paper. 
In particular, we will use without further reference that, in quandles, the actions of $\lmlt(Q)$ and $\dis(Q)$ have the same orbits. 

A binary algebraic structure satisfying the identity $(x*y)*(u*v)=(x*u)*(y*v)$ is called \emph{medial}.
A rack is medial if and only if its displacement group is abelian \cite[Proposition 2.4]{HSV}. A comprehensive study of medial quandles can be found in \cite{JPSZ1}.

Let $(Q,*)$ be a binary algebraic structure. For every $f\in\aut{Q}$ and $a\in Q$, we have 
\begin{equation}\label{L_f(a)}
L_{f(a)}=fL_af^{-1}.
\end{equation}
In particular, if $Q$ is a rack, then $L_{a*b}=L_a L_b L_a^{-1}$ for every $a,b\in Q$.

We will need the following constructions of quandles.

\begin{example}
Let $G$ be a group and $C\subseteq G$ closed under conjugation. For $a,b\in C$, let $a*b = aba^{-1}$ and $a\ld b=a^{-1}ba$. Then $(C,*,\ld)$ is a quandle, called the \emph{conjugation quandle} on $C$.
\end{example}

\begin{example}
Let $G$ be a group, $f\in\aut{G}$ and $H\le \mathrm{Fix}(f)=\setof{a\in G}{f(a)=a}$. Let $G/H$ be the set of left cosets, $G/H=\setof{aH}{a\in G}$, and define
\[  aH * bH=af(a^{-1}b)H. \]
It is easy to calculate that there is a left division operation $\ld$ such that $\mathcal{Q}_{\mathrm{Hom}}(G,H,f)=(G/H,\ast,\ld)$ is a quandle, called \emph{coset quandle}.
A coset quandle of the form $\mathcal{Q}_{\mathrm{Hom}}(G,1,f)$ is called \emph{principal}.
If, in addition, $G$ is an abelian group, then $\mathcal{Q}_{\mathrm{Hom}}(G,1,f)$ is called \emph{affine}, and it is also denoted by $\Aff(G,f)$.
\end{example}

A general quandle is called \emph{principal} (resp. \emph{affine}), if it is isomorphic to a principal (resp. affine) coset quandle.

\begin{example}
A \emph{permutation rack} is a rack whose $*$ operation does not depend on the left argument, i.e. $a*b=\sigma(b)$ where $\sigma$ is a permutation of the underlying set. In particular, by a \emph{projection quandle} we mean a permutation quandle with the operation $a*b=b$. (The adjective \emph{trivial} is reserved for one-element structures.)
\end{example}

All connected quandles with $\leq 47$ elements were enumerated \cite{HSV,RIG} and stored in the RIG library, a part of the RIG package for GAP. We often pick examples from the library, and our claims are easy to verify in GAP.
To put the subject of our study into the RIG context: there are 791 connected quandles of order $\leq47$, of which 492 are abelian, 49 nilpotent non-abelian, and 185 solvable non-nilpotent. Among the 65 non-solvable quandles, 23 are simple non-abelian, and the remaining 42 have an abelian congruence with a simple non-abelian factor.

\subsection{Congruences and homomorphisms}

Let $\alpha$ be an equivalence on a set $A$. We will use the notation $a\,\alpha\,b$ instead of $(a,b)\in\alpha$. 
The blocks will be denoted by $[a]_\alpha=\setof{b\in A}{a\,\alpha \, b}$, and we let $A/\alpha=\setof{[a]_\alpha}{a\in A}$. 
We drop the index $\alpha$ if it is clear to which equivalence we are referring to.

To study quotients (or factors) of left quasigroups, we borrow the concept of a \emph{congruence} from universal algebra \cite[Chapter 1]{Bergman}. A congruence of an algebraic structure $A$ is an equivalence $\alpha$ on $A$ compatible with all operations of $A$. For left quasigroups, this means that, for every $a,b,c,d$, 
\[ a\, \alpha \, b \text{ and }c \, \alpha \, d\ \Rightarrow\ (a*c) \, \alpha \, (b*d)\text{ and }(a\ld c)\, \alpha \, (b\ld d).\] 
Note that an equivalence $\alpha$ is compatible with a binary operation $\circ$ if and only if, for every $a,b,c$, 
\[ a\, \alpha \, b \ \Rightarrow\ (a\circ c) \, \alpha \, (b\circ c)\text{ and }(c\circ a) \, \alpha \, (c\circ b).\] 
Congruences of an algebraic structure $A$ form a complete lattice, denoted by $\Con(A)$, with the largest element $1_A=A\times A$ and the smallest element $0_A=\{(a,a):a\in A\}$. The lattice operations will be denoted by $\wedge$ and $\vee$. Namely, $\alpha\wedge \beta$ is the intersection of $\alpha$ and $\beta$, and $\alpha\vee\beta$ is the smallest congruence containing the union of $\alpha$ and $\beta$.

If $\alpha$ is a congruence of $A$, the quotient $A/\alpha$ is well defined. It is easy to see that 
\[ \Con(A/\alpha)=\{\beta/\alpha:\ \alpha\leq\beta\in\Con(A)\}\] where $[a]_\alpha \, \beta/\alpha\, [b]_\alpha$ if and only if $a\,\beta\, b$.

Let $Q$ be a left quasigroup and $\alpha$ its congruence. We will frequently use the following two observations. For every $f\in\lmlt(Q)$, if $a\, \alpha \, b$ then $f(a)\,\alpha\,f(b)$.
If $a$ is an idempotent element, then the block $[a]_\alpha$ is a subalgebra of $Q$ (indeed, if $b,c\in[a]$, then $(b*c)\,\alpha\,(a*a)=a$ and $(b\ld c)\,\alpha\,(a\ld a)=a$).

\medskip
Let $Q$ and $R$ be left quasigroups. A mapping $f:Q\to R$ is called a {\it homomorphism}, if $f(a*b)=f(a)*f(b)$ for every $a,b\in Q$. Then also $f(a\ld b)=f(a)\ld f(b)$ for every $a,b\in Q$: we have $f(b)=f(a*(a\ld b))=f(a)*f(a\ld b)$, and divide by $f(a)$ from the left.
Every homomorphism $f:Q\to R$ carries a congruence of $Q$, called the \emph{kernel}:
\[\ker f=\setof{(a,b)}{f(a)=f(b)}.\] 
By the first isomorphism theorem, $Q/\ker f\simeq \mathrm{Im}(f)$, hence quotients and homomorphic images are essentially the same thing.

Let $Q$ be a left quasigroup and $\alpha$ its congruence. It is straightforward to check that the mapping
\begin{equation}\label{pi_alpha}
\pi_{\alpha}:\lmlt (Q)\longrightarrow \lmlt(Q/\alpha),\quad L_{a_1}^{k_1}\ldots L_{a_n}^{k_n} \mapsto L_{[a_1]}^{k_1}\ldots L_{[a_n]}^{k_n}
\end{equation}
is a well defined surjective homomorphism of groups \cite{AG}.
The restriction of $\pi_\alpha$ to $\dis (Q)$ gives a surjective homomorphism $\dis (Q)\to\dis(Q/\alpha)$, and its kernel will be denoted by $\dis^\alpha$. It has the following characterization.

\begin{lemma}\label{l:kernel}
Let $Q$ be a left quasigroup and $\alpha$ its congruence. Then
\[ \dis^{\alpha}=\setof{h\in \dis(Q)}{h(a)\,\alpha\,a \text{ for every } a\in Q}. \]
\end{lemma}

\begin{proof} 
Since $[h(a)] =\pi_\alpha(h)([a]) = [a]$, then $\pi_\alpha(h)=1$ if and only if $[h(a)] = [a]$ for every $a\in Q$.
 \end{proof}


Observe that if $Q$ is a connected left quasigroup, then every factor is also connected (apply the mapping $\pi_\alpha$).
The converse is false, e.g., for any direct product of a connected and disconnected rack.

A congruence where all blocks have the same size is called \emph{uniform}.

\begin{proposition}\label{p:uniform}
Let $Q$ be a left quasigroup and $\alpha$ its congruence such that $Q/\alpha$ is connected. Then $\alpha$ is uniform.
Moreover, if $Q$ is a quandle, the blocks of $\alpha$ are pairwise isomorphic subquandles of $Q$.
\end{proposition}

\begin{proof}
Since $Q/\alpha$ is connected, for every $[a],[b]\in Q/\alpha$ there is $h\in\lmlt(Q)$ such that $[b]=\pi_\alpha(h)([a])$. Then $h|_{[a]}$ is a bijection $[a]\to[b]$: it maps $[a]$ into $[b]$, because $c\,\alpha\,a$ implies $h(c)\,\alpha\,h(a)=b$, and $h^{-1}|_{[b]}$ is its inverse mapping.

If $Q$ is a quandle then every congruence block is a subquandle, and thus $h|_{[a]}$ is an isomorphism, since $h\in\lmlt(Q)\leq \aut{Q}$. 
\end{proof}

\subsection{Orbit decomposition and Cayley kernel}

In racks, two particular congruences play a very important role: the \emph{orbit decomposition}, and the \emph{Cayley kernel}.

Let $Q$ be a left quasigroup, and $N$ a normal subgroup of $\lmlt(Q)$. We denote by $\O_N$ the transitivity relation of the action of $N$ on $Q$.
In particular, for $N=\lmlt(Q)$, we obtain the \emph{orbit decomposition} of $Q$, denoted shortly $\O_Q$.

\begin{lemma}\cite[Theorem 6.1]{BLRY}
Let $Q$ be a rack and $N\unlhd\lmlt(Q)$. Then $\O_N$ is a congruence of $Q$.
\end{lemma}

\begin{proof}
Clearly $\mathcal{O}_N$ is an equivalence relation on $Q$. Let $b \, \mathcal{O}_N \, c$, i.e. $c=f(b)$ for some $f\in N$. 
Since $N$ is normal in $\lmlt(Q)$,
\begin{align*} 
a*c & =L_a(c) = L_af(b) = L_afL_a^{-1}(a*b),\\
a\backslash c&=L_a^{-1}(c) = L_a^{-1}f(b) = L_a^{-1} fL_a (a\backslash b),
\end{align*}
and thus $(a*c)\, \mathcal{O}_N \,(a*b)$ and $(a\ld c)\, \mathcal{O}_N \,(a\ld b)$. On the other side,
\begin{align*}
 c*a&= L_c(a) = L_{f(b)}L_b^{-1}(b*a) = fL_bf^{-1}L_b^{-1}(b*a),\\
c\backslash a & =L_c^{-1} (a) = L_{f(b)}^{-1}L_b(b\backslash a) = fL_b^{-1}f^{-1}L_b(b\backslash a),
\end{align*}
and thus $(c*a)\, \mathcal{O}_N \,(b*a)$ and $(c\ld a)\, \mathcal{O}_N \,(b\ld a)$.
\end{proof}

Various properties of the $\O_N$ congruences were proved by Even and Gran in \cite{EG}: for example, that they permute with any other congruence. 

Let $Q$ be a left quasigroup. The \emph{Cayley representation} is the mapping $L_Q:Q\to\Sym(Q)$, $a\mapsto L_a$. 
For racks, $L_Q$ is a quandle homomorphism (with respect to the conjugation operation on $\Sym(Q)$), but, unlike for groups, $L_Q$ is not necessarily one-to-one. The kernel of $L_Q$, \[ \lambda_Q=\setof{(a,b)}{L_a=L_b},\] will be called the \emph{Cayley kernel} $Q$. A rack with trivial Cayley kernel is called \emph{faithful}. Note that every faithful rack is a quandle (in racks, $L_{a*a}=L_a$ for every $a$), isomorphic to a conjugation quandle (the image of the Cayley representation).

\section{Congruences and subgroups of the displacement group}\label{s:3}

\subsection{Displacement groups relative to congruences}\label{s:dis_alpha}

Let $Q$ be a left quasigroup and $\alpha$ its congruence. We define the \emph{displacement group relative to $\alpha$}, denoted by $\dis_\alpha$, as the smallest normal subgroup of $\lmlt(Q)$ containing all $L_a L_b^{-1}$ such that $a\,\alpha\,b$. That is,
\[ \dis_\alpha=\langle fL_a L_b^{-1}f^{-1}: \, a\,\alpha\,b,\,f\in\lmlt(Q)\rangle\leq\lmlt(Q).\]
The generating set is closed with respect to conjugation by any automorphism of $Q$. In particular, if $Q$ is a rack, then 
\[ \dis_\alpha=\langle L_a L_b^{-1}: \, a\,\alpha\,b\rangle.\]
The elements of the relative displacement group can be described as follows.

\begin{lemma}\label{l:words in Dis}
Let $Q$ be a left quasigroup and $\alpha$ its congruence. Then
\[ \dis_\alpha =\setof{L_{a_n}^{k_n}\ldots L_{a_1}^{k_1}L_{b_1}^{-k_1}\ldots L_{b_n}^{-k_n}} { k_i\in\mathbb{Z}, \ a_i \, \alpha \, b_i \text{ for all } i=1,\dots,n }. \]
\end{lemma}

\begin{proof}
Let $N$ denote the set on the right hand side of the expression.
Temporarily, we will say that two mappings $u,v\in\lmlt(Q)$ are \emph{$\alpha$-symmetric}, if $u=L_{a_n}^{k_n}\ldots L_{a_1}^{k_1}$ and $v=L_{b_1}^{-k_1}\ldots L_{b_n}^{-k_n}$ for some $k_i\in\mathbb{Z}$ and $a_i\,\alpha\,b_i$. So, $N$ consists of all mappings of the form $uv$ where $u,v$ are $\alpha$-symmetric.

First, we prove that $N$ is a normal subgroup of $\lmlt(Q)$. Let $f=f_1f_2$ and $g=g_1g_2$ be elements of $N$ where both $f_1,f_2$ and $g_1,g_2$ are $\alpha$-symmetric. Then the inverse $f^{-1}=f_2^{-1}f_1^{-1}$ belongs to $N$, since $f_1^{-1},f_2^{-1}$ are also $\alpha$-symmetric; the composition $fg=g_1g_1^{-1}f_1f_2g_1g_2$ belongs to $N$, since all three pairs $g_1,g_2$ and $g_1^{-1},g_1$ and $f_1,f_2$ are $\alpha$-symmetric; and the conjugate $L_a^{\pm1}fL_a^{\mp1}$ belongs to $N$ for an obvious reason.
Since $N$ contains the generators of $\dis_\alpha$, we have that $\dis_\alpha\subseteq N$.

For the other inclusion, we proceed by induction on the length of the expression, i.e., on $n=\sum_{i=1}^n |k_i|$
where $f=L_{a_n}^{k_n}\ldots L_{a_1}^{k_1}L_{b_1}^{-k_1}\ldots L_{b_n}^{-k_n}\in N$, $k_i\neq0$. 
For $n=0$, we have $f=1$ and the statement is trivial. In the induction step, let 
\[ g=L_{a_{n}}^{k_{n}-e}L_{a_{n-1}}^{k_{n-1}}\ldots L_{a_1}^{k_1}L_{b_1}^{-k_1}\ldots L_{b_{n-1}}^{-k_{n-1}}L_{b_{n}}^{-k_{n}+e}\in N,\] 
where $e=1$ if $k_n>0$, and $e=-1$ otherwise. It has a shorter length, and therefore belongs to $\dis_\alpha$. Now, since $\dis_\alpha$ is a normal subgroup,
\[ f=L_{a_n}^{e}gL_{b_n}^{-e} =
\underbrace{L_{a_n}^{e}gL_{a_n}^{-e}}_{\in\dis_\alpha}\underbrace{L_{a_n}^{e}L_{b_n}^{-e}}_{\in\dis_\alpha}\in\dis_\alpha,\]
and the proof is finished.
\end{proof}

Observe that  $\dis_\alpha\leq\dis^\alpha$: if $a\,\alpha\,b$, then $L_{[a]}L_{[b]}^{-1}$ is the identity mapping on $Q/\alpha$, and using the definition from \eqref{pi_alpha},
\[ \pi_\alpha(fL_a L_b^{-1}f^{-1})=\pi_\alpha(f)L_{[a]}L_{[b]}^{-1}\pi_\alpha(f)^{-1}=1_{Q/\alpha}.\]
 It is often the case that $\dis_\alpha\neq\dis^\alpha$. 
Obviously, this happens whenever $\alpha\leq\lambda_Q$ (hence $\dis_\alpha=1$) and $\dis(Q)\not\simeq \dis(Q/\alpha)$. 
There are also examples which cannot be explained by the Cayley kernel, e.g., in non-principal latin quandles of size $27$, as can be checked directly in the RIG library. On the positive side, $\dis_\alpha=\dis^\alpha$ in any finite principal latin quandle, see Example \ref{ex:latin_cdsg}.

\begin{proposition}\label{p:dis_alpha1}
Let $Q$ be a left quasigroup and $\alpha,\beta$ its congruences. Then
\begin{enumerate}
\item[(1)] if $\alpha\leq \beta$, then $\pi_\alpha(\dis_\beta)=\dis_{\beta/\alpha}$ and $\pi_\alpha(\dis^\beta)=\dis^{\beta/\alpha}$.
	\item[(2)] $\dis^{\alpha\wedge \beta} = \dis^{\alpha}\cap \dis^{\beta}$  and $\dis_{\alpha\vee \beta}  =  \dis_{\alpha}\dis_{\beta}$,
	\item[(3)] if $\lambda_Q$ is a congruence, then $\dis_\alpha=\dis_{\alpha\vee\lambda_Q}$,
\end{enumerate}
\end{proposition}

\begin{proof}
(1) Using Lemma \ref{l:kernel}, we have 
\begin{align*}
\dis_{\beta/\alpha}&=\langle fL_{[a]_\alpha}L_{[b]_\alpha}^{-1}f^{-1}: \, a\, \beta \, b,\, f\in\lmlt(Q/\alpha)\rangle=\pi_\alpha(\dis_\beta),\\
\dis^{\beta/\alpha}&=\setof{\pi_\alpha(h)\in \dis(Q/\alpha)}{h(a)\, \beta \, a}=\pi_\alpha(\dis^\beta).
\end{align*}

(2) For intersection, using Lemma \ref{l:kernel},
\begin{align*}
\dis^{\alpha\wedge\beta}&=\setof{h\in \dis(Q)}{h(a) \,(\alpha\wedge\beta)\, a\, , \forall a\in Q}\\
&=\setof{h\in \dis(Q)}{h(a) \, \alpha \, a\,\text{ and }\, h(a)\,\beta \, a , \forall a\in Q}=\dis^\alpha\cap\dis^\beta.
\end{align*}
For join, it is easy to see that both $\dis_\alpha, \dis_\beta\leq \dis_{\alpha\vee\beta}$, and thus $\dis_\alpha\dis_\beta\leq \dis_{\alpha\vee\beta}$. 
For the other inclusion, let $a\,(\alpha\vee\beta)\,b$, and take the witnesses $a=a_1,\dots,a_n$ and $b_1,\dots,b_n=b$ such that $a_i\,\alpha\,b_i$ and $b_i\,\beta\,a_{i+1}$, for every $i$. Then
\[ L_aL_b^{-1}=\underbrace{L_{a_1}L_{b_1}^{-1}}_{\in\dis_\alpha}\underbrace{L_{b_1}L_{a_2}^{-1}}_{\in\dis_\beta}\underbrace{L_{a_2}L_{b_2}^{-1}}_{\in\dis_\alpha}\cdots\underbrace{L_{a_n}L_{b_n}^{-1}}_{\in\dis_\beta}\in\dis_\alpha\dis_\beta, \]
and thus every generator $fL_aL_b^{-1}f^{-1}$ of $\dis_{\alpha\vee\beta}$ belongs to $\dis_\alpha\dis_\beta$.

(3) This is an immediate consequence of (2) for $\beta=\lambda_Q$, since $\dis_{\lambda_Q}=1$.
%
\end{proof}

\begin{proposition}\label{p:dis_alpha2}
Let $Q$ be a rack and $\alpha$ its congruence. Then
\begin{enumerate}
	\item[(1)] $[\dis^{\alpha},\lmlt(Q)]\leq\dis_{\alpha}$,
	\item[(2)] $\dis^\alpha=\dis^{\mathcal{O}_{\dis^\alpha}}$ and $\mathcal{O}_N=\mathcal{O}_{\dis^{\mathcal{O}_N} }$,
	\item[(3)] $\mathcal O_{\dis^\alpha}\leq\alpha$,
	\item[(4)] if $Q$ is a quandle and $\dis_\alpha=\dis(Q)$, then $\O_Q\leq\alpha$.
\end{enumerate}
\end{proposition}

\begin{proof}
(1) Consider $f\in\dis^{\alpha}\leq\aut{Q}$ and $a\in Q$. Then $f(a)\,\alpha\,a$, and thus, using \eqref{L_f(a)}, $[f,L_a]=L_aL^{-1}_{f(a)}\in\dis_\alpha$.

(2) Let $\beta=\mathcal{O}_{\dis^\alpha}$. Then $\beta\leq\alpha$, and so $\dis^\beta\leq \dis^\alpha$. In the other direction, for $h\in \dis^\alpha$ we have $h(a) \, \beta \, a$ for every $a\in Q$, and thus $h\in \dis^\beta$. \\
According to Lemma \ref{l:kernel}, $N\leq \dis^{\mathcal{O}_N}$ and the orbits of $\dis^{\mathcal{O}_N}$ are contained in the orbit of $N$. Therefore $N$ and $\dis^{\mathcal{O}_N}$ have the same orbits, i.e. $\mathcal{O}_N=\mathcal{O}_{\dis^{\mathcal{O}_N}}$.

(3) If $b=h(a)$ for some $h\in\dis^\alpha$, then $b=h(a)\,\alpha\,a$ by Lemma \ref{l:kernel}.

(4) If $\dis_\alpha=\dis(Q)$, then also $\dis^\alpha=\dis(Q)$ and thus $\dis(Q/\alpha)=1$. So $Q/\alpha$ is a projection quandle and thus $\O_Q\leq \alpha$.
\end{proof}

The converse of (4) fails, for example, for any 2-reductive medial quandle $Q$ which is not a projection quandle: there we have $\mathcal{O}_Q\leq\lambda_Q$ and thus $\dis_{\O_Q}=1$ (see \cite{JPSZ1} for details). In condition (4), the assumption of idempotence is necessary: for example, permutation racks have trivial displacement groups, but $\mathcal O_Q$ can be non-trivial.

\subsection{Congruences determined by subgroups}

Let $Q$ be a left quasigroup. We will denote $\N(Q)$ the lattice of all subgroups of $\dis(Q)$ that are normal in $\lmlt(Q)$ (this is a sublattice of the normal subgroups of $\dis(Q)$). For $N\in\N(Q)$, we define a relation
\[ \con N=\{(a,b):L_aL_b^{-1}\in N\},\]
called the \emph{equivalence determined by $N$}.

\begin{lemma}
Let $Q$ be a rack and $N\in\N(Q)$. Then $\con N$ is a congruence of $Q$.
\end{lemma}

\begin{proof}
Assume $a\, \con N \, b$, i.e., $L_aL_b^{-1}\in N$, and let $c\in Q$. Since $N$ is normal in $\lmlt(Q)$,
\[ L_{c*a}L_{c*b}^{-1}=L_cL_aL_c^{-1}L_cL_b^{-1}L_c^{-1}=L_cL_aL_b^{-1}L_c^{-1} \in N,\]
hence $(c*a)\, \con N \, (c*b)$, and similarly, $(c\ld a)\, \con N \, (c\ld b)$. On the other hand,
\[ L_{a*c}L_{b*c}^{-1} = (L_a L_c L_a^{-1})( L_b L_c^{-1} L_b^{-1}) = \underbrace{(L_a L_b^{-1})}_{\in N} \underbrace{(L_b L_c (L_a^{-1} L_b) L_c^{-1} L_b^{-1})}_{\in N}\in N, \]
hence $(a*c)\, \con N \, (b*c)$, and similarly, $(a\ld c)\, \con N \, (b\ld c)$.
\end{proof}

\begin{proposition}\label{p:con_N}
Let $Q$ be a rack and and $N\in\N(Q)$. Then
\begin{enumerate}
	\item $\dis_{\con N}\leq N\leq\dis^{\con N}$,
	\item $\con{N}=1_Q$ if and only if $N=\dis(Q)$,
	\item if $N=\bigcap_{i\in I} N_i$, $N_i\in\N(Q)$, then $\con{N}=\bigwedge_{i\in I} \con{N_i}$,
	\item $\O_N\leq \con N$.
\end{enumerate}
\end{proposition}

\begin{proof}
(1) For the first inequality, note that $\dis_{\con N}$ is generated by all pairs $L_aL_b^{-1}$ which belong to $N$.
For the second inequality, let $h\in N$. For every $a\in Q$, we have 
\begin{displaymath}
L_{h(a)} L_a^{-1}= h L_a h^{-1} L_a^{-1}\in N,
\end{displaymath}
because $N$ is normal in $\lmlt(Q)$, and thus $h(a) \, \con N \, a$. Using Lemma \ref{l:kernel}, $h\in\dis^{\con N}$.

(2) Clearly $\con{\dis(Q)}=1_Q$. If $\con{N}=1_Q$, then $\dis(Q)=\dis_{1_Q}=\dis_{\con N}\leq N$ by (1).

(3) Clearly $L_a L_b^{-1}\in N$ if and only if $L_aL_b^{-1}\in N_i$ for every $i\in I$.

(4) If follows from item (1) using Lemma \ref{l:kernel}.
\end{proof}

As we shall see soon, the $\dis$ and $\con{}$ operators form a monotone Galois connection between $\Con(Q)$, the congruence lattice of a rack $Q$, and the lattice $\N(Q)$.
Note that $\lambda_Q\leq\con N$ for every $N$, and if $\alpha\leq\lambda_Q$ then $\dis_\alpha=1$. Therefore, the smaller the Cayley kernel is, the finer properties of the connection one can expect.

\subsection{A Galois connection}

Recall that a \emph{monotone Galois connection} is a pair of monotone functions between two ordered sets, $F:X\to Y$, $G:Y\to X$, such that $F(x)\leq y$ if and only if $x\leq G(y)$. Then, $GF$ is a closure operator  on $X$, $FG$ is a kernel operator on $Y$, and $FGF=F$ and $GFG=G$. (See \cite[Section 2.5]{Bergman} for details.)

\begin{proposition}\label{galois_connection}
Let $Q$ be a rack. Then $\alpha\mapsto\dis_\alpha$ and $N\mapsto\con N$ is a monotone Galois connection between $\Con(Q)$ and $\N(Q)$.
\end{proposition}

\begin{proof}
Both mappings are indeed monotone. We prove that $\dis_\alpha\leq N$ if and only if $\alpha\leq \con N$.

$(\Rightarrow)$ 
If $a\,\alpha\,b$, then $L_aL_b^{-1}\in\dis_\alpha\subseteq N$, and thus $a\,\con N\, b$.

$(\Leftarrow)$ 
We need to show that $L_a L_b ^{-1} \in N$ whenever $a\, \alpha \, b$. But $a\,\alpha\,b$ implies $a\,\con N\, b$, and thus $L_aL_b^{-1}\in N$ by definition.
\end{proof}

In particular, the closure property says that $\alpha\leq\con {\dis_{\alpha}}$, and the kernel property says that $\dis_{\con N}\leq N$.
 
For a given rack $Q$, the connection of Proposition \ref{galois_connection} is rarely bijective. The Cayley kernel is one obvious reason: the $\con{}$ operator cannot reach congruences below $\lambda_Q$, and the $\dis$ operator maps all congruences below $\lambda_Q$ to the trivial subgroup. Other reasons must exist, too, since, for example, neither operator is 1-1 or onto in the non-principal latin quandles of order $27$ mentioned earlier.


The connection of Proposition \ref{galois_connection} can recognize certain properties of factors. For example, faithfulness.

\begin{proposition}\label{p:con_dis}
Let $Q$ be a rack and $\alpha$ its congruence. Then $Q/\alpha$ is faithful if and only if $\alpha=\con {\dis^{\alpha}}$.
\end{proposition}

\begin{proof}
We prove that $\con{\dis^{\alpha}}/\alpha=\lambda_{Q/\alpha}$. Indeed, for $[a],[b]\in Q/\alpha$, 
\begin{displaymath}
[a]\, (\con{\dis^{\alpha}}/\alpha) \, [b] \ \Longleftrightarrow \ a\, \con{\dis^{\alpha} } \, b \ \Longleftrightarrow \ L_a L_b^{-1}\in \dis^{\alpha}   \ \Longleftrightarrow \ L_{[a]}=L_{[b]}.
\end{displaymath}
Therefore, $Q/\alpha$ is faithful if and only if $\con{\dis^\alpha }/\alpha=0_{Q/\alpha}$, that is, if and only if $\alpha=\con {\dis^{\alpha}}$.
\end{proof}

Proposition \ref{p:con_dis} implies that if $Q/\alpha$ is faithful and $\dis_\alpha\neq \dis^\alpha$, then $\alpha=\con{\dis_\alpha}=\con{\dis^\alpha}$, so the $\con{}$ operator cannot be injective.

\begin{proposition}
Let $Q$ be a quandle such that every factor of $Q$ is faithful. Then the $\dis$ operator is injective and the $\con{}$ operator is surjective.
\end{proposition}

\begin{proof}
By Proposition \ref{p:con_dis}, we have $\alpha\leq\con{\dis_\alpha}\leq\con{\dis^\alpha}=\alpha$, and thus $\con{\dis_\alpha}=\alpha$. This means that the $\con{}$ operator is left inverse to the $\dis$ operator. Hence $\dis$ must be injective and $\con{}$ must be surjective.
\end{proof}

\begin{remark}
Let $Q$ be a rack. Then $\alpha \mapsto  \dis^\alpha$ and $N\mapsto \mathcal{O}_N $ is also a monotone Galois connection between $\Con(Q)$ and $\N(Q)$. Indeed, from the characterization of $\dis^\alpha$ in Lemma \ref{l:kernel} it is easy to check that $N\leq \dis^\alpha$ if and only if $\mathcal{O}_N\leq \alpha$. At the moment, we have no use for this connection, and therefore focus on the one from Proposition \ref{galois_connection}. 
\end{remark}

\subsection{Quandles with congruences determined by subgroups}

A rack is said to have \emph{congruences determined by subgroups} (shortly, \emph{CDSg}), if the connection from Proposition \ref{galois_connection} is a lattice isomorphism $\Con(Q)\simeq\N(Q)$. 

\begin{example}
All simple quandles of size $>2$ have CDSg, since both lattices have just two elements, as proved in \cite[Lemma 2]{Joyce-simple}.
\end{example}

\begin{example}\label{ex:latin_cdsg}
Finite affine latin quandles have CDSg. They are polynomially equivalent to modules, and therefore, congruences correspond to submodules, which are exactly the elements of $\N(Q)$ (see \cite[Section 2.3]{Sta-latin} for an explanation of polynomial equivalence in the context of quasigroups).
A similar argument works more generally, for all principal latin quandles, which are polynomially equivalent to algebraic structures of the form $(\dis(Q),\cdot,^{-1},1,\widehat{L_e})$ where $\widehat{L_e}$ denotes the inner automorphism given by $L_e$.
\end{example}

In the rest of the section, we characterize racks with CDSg and show that they are actually connected quandles.

\begin{lemma}\label{star_under_H}
Let $Q$ be a rack and $\alpha$ its congruence. If $Q$ has CDSg, then $Q/\alpha$ has CDSg.
\end{lemma}

\begin{proof}
We know that $\Con(Q/\alpha)\simeq [\alpha,1_Q]$ and $\N(Q/\alpha)\simeq [\dis^\alpha,\dis(Q)]$. The mapping $\dis$ restricted to this interval is an isomorphism as well (and restricted $\con{}$ is its inverse), and the following diagram is commutative:
\begin{center}
$\xymatrixcolsep{8pc}\xymatrix{
[\alpha,1_Q]  \ar[r]^{\dis}\ar[d] & [\dis^\alpha,\dis(Q)] \ar[d] \\
\Con(Q/\alpha) \ar[r]^{\dis}& \N(Q/\alpha)
}
$
\end{center} 
(the vertical arrows are the canonical isomorphisms).
\end{proof} 

\begin{proposition}\label{p:CDSg}
Let $Q$ be a rack. Then $Q$ has CDSg if and only if the following two conditions hold:
\begin{enumerate}
\item $\dis_\alpha=\dis^\alpha$ for every $\alpha\in\Con(Q)$,
\item every factor of $Q$ is faithful. 
\end{enumerate}
\end{proposition}

\begin{proof}
$(\Rightarrow)$ Since $\dis_{0_Q}=\dis_{\lambda_Q}=1$, necessarily $\lambda_Q=0_Q$, and $Q$ is faithful. Hence also every factor of $Q$ is faithful, and since $\alpha=\con {\dis^\alpha}=\con {\dis_\alpha}$, we have $\dis_\alpha=\dis^\alpha$ for every congruence $\alpha$.

$(\Leftarrow)$ Let $N\in\N(Q)$. Then $\dis_{\con{N}}\leq N \leq \dis^{\con{N}}$, and condition (1) implies that $N=\dis_{\con{N}}$. By Proposition \ref{p:con_dis}, $\con{\dis_\alpha}=\alpha$ for every $\alpha\in\Con(Q)$.
\end{proof}

\begin{proposition}
Let $Q$ be a rack with CDSg. Then $\alpha=\mathcal{O}_{\dis_\alpha}$ for every $\alpha\in\Con(Q)$. In particular, $Q$ is a connected quandle.
\end{proposition}

\begin{proof}
In racks, $L_a=L_{a*a}$ for every $a$, hence non-idempotent racks are never faithful, contradicting Proposition \ref{p:CDSg}(2).
Let $\beta=\mathcal{O}_{\dis_\alpha}$. By Proposition \ref{p:CDSg}(1), $\dis_\alpha=\dis^\alpha$ and $\dis_\beta=\dis^\beta$, hence Proposition \ref{p:dis_alpha2}(2) says that $\dis_\alpha=\dis_\beta$, and thus $\dis_{\alpha/\beta}=1$ and $\alpha/\beta\leq \lambda_{Q/\beta}$. But $Q/\beta$ is faithful, hence $\alpha=\beta$. In particular, for $\alpha=1_Q$ we obtain that $Q$ is connected.
\end{proof}

\section{Universal algebraic concepts}\label{s:4}

\subsection{Left translation terms}\label{s:ltt}

A term is a well-defined expression using variables and operation symbols in a given language; we will write $t(x_1,\dots,x_n)$ for a term using a subset of variables $x_1,\dots,x_n$. Given a term $t(x_1,\dots,x_n)$ and an algebraic structure $A=(A,\ldots)$, the associated \emph{term function} $t^A:A^n\to A$ evaluates the term $t$ in $A$. (See \cite[Section 4.3]{Bergman} for formal definitions).

From now on, we will use the language $\{\ast,\backslash\}$ of left quasigroups. Terms will be considered as labeled rooted binary trees, with inner nodes labeled by operations and leaves by variables. 

\emph{Left translation terms} (shortly, \emph{lt-terms}) are the terms in the language of left quasigroups of the form
\begin{equation}\label{ltt}
t(x_1,\ldots,x_n)=s_1(x_{i_1})\circ_1(s_2(x_{i_2})\circ_2(\ldots(s_{i_m}(x_{i_{m-1}})\circ_{m-1}s_{m}(x_{i_{m}})))),
\end{equation}
where $\circ_j\in\{\ast,\backslash\}$ and $s_i$ is a unary term for $i_j\in\{1,\dots,n\}$. Somewhat less formally, we can write
\[ t(x_1,\ldots,x_n)=L_{s_1(x_{i_1})}^{\eps_1}\cdots L_{s_{m-1}(x_{i_{m-1}})}^{\eps_{m-1}}(s_m(x_{i_{m}})) \]
where $\eps_j=1$ if $\circ_j=*$ and $\eps_j=-1$ if $\circ_j=\ld$
(the expression makes a formal sense in the free left quasigroup over the alphabet $x_1,\dots,x_n$).

Note that the left multiplication group of a left quasigroup consists of all term functions resulting from lt-terms for which all the unary terms are trivial (i.e. $s_j(x_{i_j})=x_{i_j}$). Then, somehow, lt terms are obtained combining unary terms and the elements of the left multiplication group.

A left quasigroup, in which every term function results from some lt-term, will be called \emph{ltt} left quasigroup.
The following fact is well known and crucial for our adaptation of the general commutator theory to racks.

\begin{proposition}
Every rack is an ltt left quasigroup.
\end{proposition}

\begin{proof}
Using equation \eqref{L_f(a)}, it is easy to see that the following identites hold in every rack:
\begin{eqnarray*}
(x\ast y)\ast z=x\ast (y\ast (x\ld z), &\quad & (x\ld y)\ast z=x\ld (y\ast (x\ast z), \\
(x\ast y)\ld z=x\ast (y\ld (x\ld z), &\quad & (x\ld y)\ld z=x\ld (y\ld (x\ld z),
\end{eqnarray*}
Using these identities, every term can be transformed to an lt-term, by repeatedly expanding the uppermost left subterm which is not a leaf (see Figure \ref{fig:ltt} for an example). 
\end{proof}

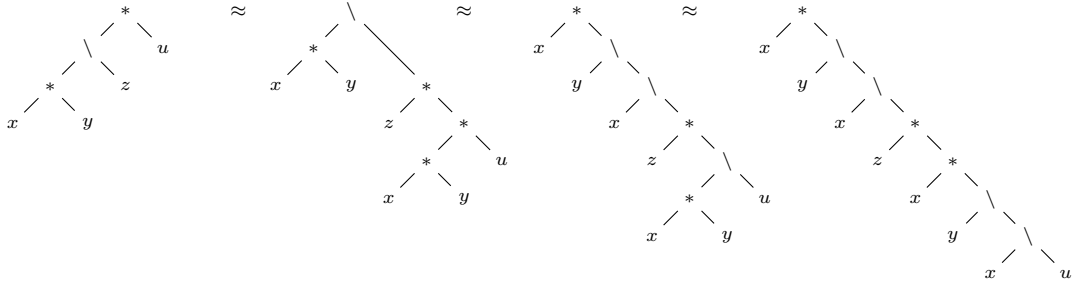
\begin{figure}\label{fig:ltt}

\tiny{
\begin{tikzpicture}[scale =0.5]
\node(a0) at (0,0) {$\ast$};  \node(m) at (3,0) {$\approx$}; 
\node(a1) at (-1,-1) {$\backslash$}; \node(a2) at (1,-1) {$u$}; 
\node(b1) at (-2,-2) {$\ast$}; \node(b2) at (0,-2) {$z$}; 
\node(c1) at (-3,-3) {$x$}; \node(c2) at (-1,-3) {$y$}; 

\node(aa0) at (6,0) {$\backslash$}; 
\node(aa1) at (5,-1) {$\ast$};
\node(bb1) at (4,-2) {$x$}; \node(bb2) at (6,-2) {$y$};
 \node(cc1) at (8,-2) {$\ast$};
\node(d1) at (7,-3) {$z$};  \node(d2) at (9,-3) {$\ast$};
\node(e1) at (8,-4) {$\ast$};  \node(e2) at (10,-4) {$u$};
\node(f1) at (7,-5) {$x$};  \node(f2) at (9,-5) {$y$};

\draw(aa0) -- (aa1);
\draw(aa1) -- (bb1);
\draw(aa0) -- (cc1);
\draw(cc1) -- (d1);

\draw(e1) -- (f1);

\draw(e1) -- (f2);

\draw(cc1) -- (d2);
\draw(d2) -- (e1);
\draw(d2) -- (e2);

\draw(aa1) -- (bb2);
\draw(a0) -- (a1);
\draw(a0) -- (a2);
\draw(a1) -- (b1);
\draw(a1) -- (b2);
\draw(b1) -- (c1);
\draw(b1) -- (c2);

\node(a0) at (12,0) {$\ast$}; \node at (9,0) {$\approx$};
\node(b0) at (11,-1) {$x$}; \node(b1) at (13,-1) {$\backslash$};
\node(c0) at (12,-2) {$y$}; \node(c1) at (14,-2) {$\backslash$};
\node(dd0) at (13,-3) {$x$}; \node(dd1) at (15,-3) {$\ast$};
\node(ee0) at (14,-4) {$z$}; \node(ee1) at (16,-4) {$\backslash$};
\node(ff0) at (15,-5) {$\ast$}; \node(ff1) at (17,-5) {$u$};
\node(gg0) at (14,-6) {$x$}; \node(gg1) at (16,-6) {$y$};

\draw(a0) -- (b0);
\draw(a0) -- (b1);
\draw(b1) -- (c0);
\draw(b1) -- (c1);
\draw(c1) -- (dd0);
\draw(c1) -- (dd1);
\draw(dd1) -- (ee0);
\draw(dd1) -- (ee1);
\draw(ee1) -- (ff0);
\draw(ee1) -- (ff1);
\draw(ff0) -- (gg0);
\draw(ff0) -- (gg1);

\node(f) at (18,0) {$\ast$}; \node at (15,0) {$\approx$};
\node(b1) at (17,-1) {$x$}; \node(b2) at (19,-1) {$\backslash$};
\node(b3) at (18,-2) {$y$}; \node(b4) at (20,-2) {$\backslash$};
\node(b5) at (19,-3) {$x$}; \node(b6) at (21,-3) {$\ast$};
\node(b7) at (20,-4) {$z$}; \node(b8) at (22,-4) {$\ast$};
\node(b9) at (21,-5) {$x$}; \node(b10) at (23,-5) {$\backslash$};
\node(b11) at (22,-6) {$y$}; \node(b12) at (24,-6) {$\backslash$};
\node(b13) at (23,-7) {$x$}; \node(b14) at (25,-7) {$u$};

\draw(f) -- (b1);
\draw(f) -- (b2);
\draw(b2) -- (b3);
\draw(b2) -- (b4);
\draw(b4) -- (b5);
\draw(b4) -- (b6);
\draw(b6) -- (b7);
\draw(b6) -- (b8);
\draw(b8) -- (b9);
\draw(b8) -- (b10);
\draw(b10) -- (b11);
\draw(b10) -- (b12);
\draw(b12) -- (b13);
\draw(b12) -- (b14);

\end{tikzpicture} }
\caption{Transforming the term $((x*y)\ld z)*u$ into a left translation form.}
\end{figure}

The proof suggests that there are many more examples of ltt left quasigroups: any quadruple of identities will work, as long as it ``flattens" the term. For instance, A different class of ltt left quasigroups can be defined by modifying self-distributivity as follows.

\begin{example}
Let $\mathcal C$ be the class of left quasigroups satisfying the identities
\[ (x\ast y)\ast z = x\backslash (y\backslash (x\ast z)) \quad\text{ and }\quad (x\backslash y) \ast z = x\ast (y \backslash (x\backslash z)),\]
resulting from the corresponding rack identities by switching $*$ and $\ld$ on the right-hand side. The ltt property can be proved similarly as for racks.

For involutory left quasigroups, the identities for $\mathcal C$ are equivalent to self-distributivity. However, in $\mathcal C$, $L_{a*a}=L_a^{-1}$, while in racks, $L_{a*a}=L_a$. Hence, the racks in $\mathcal C$ are precisely the involutory racks.
An exhaustive computer search reveals that the smallest member of $\mathcal C$ which is not a rack has 4 elements and it is unique up to isomorphism. Further numbers of members of $\mathcal C$ which are not a rack are 1 of order 5, 8 of order 6, 20 of order 7 and 125 of order 8. The multiplication tables of the two smallest examples are below.

\begin{center}
\begin{tabular}{|cccc| }
\hline
2 & 4 &1 & 3 \\ 
3 & 1 & 4 &2 \\
3 & 1 & 4 &2 \\
2 & 4 &1 &3  \\
\hline
\end{tabular}\qquad\qquad \begin{tabular}{|ccccc|}
\hline
3 & 4 &1 & 5 & 2 \\ 
3 & 2 & 1 &4 & 5 \\
3 & 5 & 1 &2 & 4 \\
3 & 2 &1 &4 & 5  \\
3 & 2 &1 &4 & 5  \\
\hline
\end{tabular}
\end{center}
\end{example}
 
\subsection{The commutator theory}\label{sec:commutator}

Let $A$ be an algebraic structure. The commutator is a binary operation on the lattice $\Con(A)$, defined using the concept of centralization, explained below.
For further study, we refer to \cite{FM,MS}.

Let $\alpha$, $\beta$, $\delta$ be congruences of $A$. We say that \emph{$\alpha$ centralizes $\beta$ over $\delta$}, and write $C(\alpha,\beta;\delta)$, if for every $(n+1)$-ary term operation $t$, every pair $a\,\alpha\,b$ and every $u_1\,\beta\,v_1$, $\dots$, $u_n\,\beta\,v_n$ we have
\[  t^A(a,u_1,\dots,u_n)\,\delta \, t^A(a,v_1,\dots,v_n)\quad\text{implies}\quad t^A(b,u_1,\dots,u_n) \, \delta \, t^A(b,v_1,\dots,v_n). \tag{TC}\]
The implication (TC) is referred to as the \emph{term condition} for $t$, or shortly $\TC(t,\alpha,\beta,\delta)$. 
It is easy to show that $C(\alpha,\beta;\delta)$ holds if and only if $TC(t,\alpha,\beta,\delta)$ is satisfied for every term $t$ in which the first variable occurs only once: indeed, we can use (TC) several times to replace every occurrence one-by-one (see \cite[Lemma 4.1]{SV1} for a formal proof). 
We will also need the following observations:
\begin{itemize}
	\item[(C1)] if $C(\alpha,\beta;\delta_i)$ for every $i\in I$, then $C(\alpha,\beta;\bigwedge\delta_i)$,
	\item[(C2)] $C(\alpha,\beta;\alpha\wedge\beta)$,
	\item[(C3)] if $\theta\leq\alpha\wedge\beta\wedge\delta$, then $C(\alpha,\beta;\delta)$ in $A$ if and only if $C(\alpha/\theta,\beta/\theta;\delta/\theta)$ in $A/\theta$.
\end{itemize}

Now, the \emph{commutator} of $\alpha$, $\beta$, denoted by $[\alpha,\beta]$, is the smallest congruence $\delta$ such that $C(\alpha,\beta;\delta)$ (the definition makes sense thanks to (C1)). From (C2) follows that $[\alpha,\beta]\leq\alpha\wedge\beta$.
Finally, a congruence $\alpha$ is called 
\begin{itemize}
	\item \emph{abelian} if $C(\alpha,\alpha;0_A)$, i.e., if $[\alpha,\alpha]=0_A$,
	\item \emph{central} if $C(\alpha,1_A;0_A)$, i.e., if $[\alpha,1_A]=0_A$.
\end{itemize}
Subsequently, the \emph{center} of $A$, denoted by $\zeta_A$, is the largest congruence of $A$ such that $C(\zeta_A,1_A;0_A)$. Hence, $[\alpha,1_A]$ is the smallest congruence $\delta$ such that $\alpha/\delta\leq\zeta_{A/\delta}$.
Similarly, $[\alpha,\alpha]$ is the smallest congruence $\delta$ such that $\alpha/\delta$ is an abelian congruence of $A/\delta$. 

The following lemma resembles the second isomorphism theorem for groups, and will be used later in induction arguments.

\begin{lemma}\label{l:2rd_iso_thm}
Let $A$ be an algebraic structure, and $\theta\leq\alpha\leq\beta$ its congruences. 
Then $\beta/\alpha$ is central (resp. abelian) in $A/\alpha$ if and only if $(\beta/\theta)\big/(\alpha/\theta)$ is central (resp. abelian) in $(A/\theta)\big/(\alpha/\theta)$.
\end{lemma}

\begin{proof}
In case of centrality, using (C3) repetitively, we obtain
\begin{align*}
C(\beta/\alpha,1_{A/\alpha};0_{A/\alpha})\text{ in } A/\alpha 
\ &\Leftrightarrow\ 
C(\beta,1_A;\alpha)\text{ in } A  \\
&\Leftrightarrow\ 
C(\beta/\theta,1_{A/\theta};\alpha/\theta)\text{ in } A/\theta \\
&\Leftrightarrow\ 
C((\beta/\theta)\big/(\alpha/\theta),1_{(A/\theta)\big/(\alpha/\theta)};0_{(A/\theta)\big/(\alpha/\theta)}))\text{ in } (A/\theta)\big/(\alpha/\theta).
\end{align*}
A similar argument works for abelianness, too.
\end{proof}

An algebraic structure $A$ is called \emph{abelian} if $\zeta_A=1_A$, or, equivalently, if the congruence $1_A$ is abelian.
It is called \emph{nilpotent} (resp. \emph{solvable}) if and only if there is a chain of congruences 
\begin{displaymath}
    0_A=\alpha_0\leq\alpha_1\leq\ldots\leq\alpha_n=1_A
\end{displaymath}
such that $\alpha_{i+1}/\alpha_{i}$ is a central (resp. abelian) congruence of $A/\alpha_{i}$, for all $i\in\{0,1,\dots,n-1\}$.
The length of the smallest such series is called the \emph{length} of nilpotence (resp. solvability).

Similarly to group theory, one can define the series
\begin{displaymath}
    \gamma_{0}=1_A,\qquad \gamma_{i+1}=[\gamma_{i},1_A],
\end{displaymath}
and
\begin{displaymath}
    \gamma^{0}=1_A,\qquad \gamma^{i+1}=[\gamma^{i},\gamma^{i}],
\end{displaymath}
and prove that an algebraic structure $A$ is nilpotent (resp. solvable) of length $n$ if and only if $\gamma_{n}=0_A$ (resp. $\gamma^n=0_A$). Note that both definitions use a special type of commutators: nilpotence uses commutators $[\alpha,1_A]$, solvability uses commutators $[\alpha,\alpha]$. 

In groups, the commutator and the corresponding notions of abelianness and centrality coincide with the classical terminology. 
In loops, the situation is more complicated \cite{SV1}. In a wider setting, the commutator behaves well in all congruence-modular varieties \cite{FM}; for example, it is commutative (note that its definition is asymmetric with respect to $\alpha,\beta$). For racks, the commutator in general lacks many desired properties, such as commutativity (Proposition \ref{p:comm_comm} and Example \ref{ex:non-commutative commutator}), nevertheless, the notions of abelianness and centrality seem to have a very good meaning.

We also point out that there is no general principle providing a natural generating set for the congruence commutator, such as the element-wise commutators in groups. Analogies are known in several special cases, including loops \cite{SV1} and quasigroups \cite{Bel}.

\section{The commutator in racks}\label{s:5}

\subsection{From universal algebra to racks}\label{s:commutators}

The crucial fact is that in racks, or more generally, in ltt left quasigroups, the centralizing relation for congruences is related to the properties of the corresponding relative displacement groups.

\begin{lemma}\label{l:C(ab0)}
Let $Q$ be a left quasigroup, $\alpha,\beta$ its congruences, and consider the following conditions:
\begin{enumerate}
	\item $C(\alpha,\beta;0_Q)$;
	\item $[\dis_{\alpha},\dis_{\beta}]=1$ and $\dis_\beta$ acts $\alpha$-semiregularly on $Q$.
\end{enumerate}
Then (1) implies (2), and if $Q$ is an ltt left quasigroup then (2) implies (1).
\end{lemma}

\begin{proof}
$(1)\Rightarrow(2)$.
For the first property, it is sufficient to show that generators of the groups $\dis_{\alpha}$ and $\dis_{\beta}$ commute. Let $a\,\alpha\,b$ and $c\,\beta\,d$. Let $f=L_{u_1}^{k_1}\cdots L_{u_m}^{k_m}$ and $g=L_{v_1}^{l_1}\cdots L_{v_n}^{l_n}$ be elements of $\lmlt(Q)$. We want to show that 
\[ (fL_{a}L_{b}^{-1}f^{-1})(gL_{c}L_{d}^{-1}g^{-1})=(gL_{c}L_{d}^{-1}g^{-1})(fL_{a}L_{b}^{-1}f^{-1}).\] 
Denote by $Y$ the formal expression $L_{y_1}^{k_1}\cdots L_{y_m}^{k_m}$, and $Z$ the formal expression $L_{z_1}^{l_1}\cdots L_{z_n}^{l_n}$, with inverses defined accordingly, and consider the $(m+n+5)$-ary term
\[t(x_0,x_1,x_2,x_3,x_4,\bar y,\bar z)=ZL_{x_2}L_{x_3}^{-1}Z^{-1}YL_{x_0}L_{x_1}^{-1}Y^{-1}ZL_{x_3}L_{x_2}^{-1}Z^{-1}(x_4).\]
Then, for any $e\in Q$, we have
\begin{align*}
t^Q(b,b,c,d,e,\bar u,\bar v)&=gL_{c}L_{d}^{-1}g^{-1}fL_{b}L_{b}^{-1}f^{-1}gL_{d}L_{c}^{-1}g^{-1}(e) \\ &=e \\
&=gL_{c}L_{c}^{-1}g^{-1}fL_{b}L_{b}^{-1}f^{-1}gL_{c}L_{c}^{-1}g^{-1}(e)=t^Q(b,b,c,c,e,\bar u,\bar v). 
\end{align*}
Using $TC(t,\alpha,\beta,0_Q)$, we replace $a$ for $b$ and obtain
\begin{align*}
t^Q(a,b,c,d,e,\bar u,\bar v)&=gL_{c}L_{d}^{-1}g^{-1}fL_{a}L_{b}^{-1}f^{-1}gL_{d}L_{c}^{-1}g^{-1}(e) \\
&=gL_{c}L_{c}^{-1}g^{-1}fL_{a}L_{b}^{-1}f^{-1}gL_{c}L_{c}^{-1}g^{-1}(e)=t^Q(a,b,c,c,e,\bar u,\bar v). 
\end{align*}
By all possible choices of $e$ we obtain
\[ gL_{c}L_{d}^{-1}g^{-1}fL_{a}L_{b}^{-1}f^{-1}gL_{d}L_{c}^{-1}g^{-1}=fL_{a}L_{b}^{-1}f^{-1},\] 
which was our goal.

Next, we show the semiregularity property. Lemma \ref{l:words in Dis} says that 
$\dis_\beta$ consists of all mappings $L_{a_n}^{k_n}\ldots L_{a_1}^{k_1}L_{b_1}^{-k_1}\ldots L_{b_n}^{-k_n}$ such that
$k_i\in\{\pm1\}$ and $a_i\,\beta\,b_i$ for all $i=1,\dots,n$. Given a mapping $g\in\dis_\beta$ in this form, consider the $(2n+1)$-ary term
\[t(x_0,x_1,\dots,x_{2n})=L_{x_n}^{k_n}\ldots L_{x_1}^{k_1}L_{x_{n+1}}^{-k_1}\ldots L_{x_{2n}}^{-k_n}(x_0).\]
Now, assume that for some $a\in Q$
\[t^Q(a,a_1,\dots,a_n,b_1,\dots,b_n)=g(a)=a=t^Q(a,a_1,\dots,a_n,a_1,\dots,a_n).\]
Using $TC(t,\alpha,\beta,0_Q)$, we can replace $a$ for an arbitrary $b$ such that $b\,\alpha\,a$ and obtain
\[t^Q(b,a_1,\dots,a_n,b_1,\dots,b_n)=g(b)=b=t^Q(b,a_1,\dots,a_n,a_1,\dots,a_n).\]

$(2)\Rightarrow(1)$.
It is sufficient to verify $TC(t,\alpha,\beta,0_Q)$ for every term $t(x_0,\dots,x_n)$ as 

$$t(x_1,\ldots,x_n)=L_{s_1(y_1)}^{k_1}\ldots L_{s_m(y_m)}^{k_m}(s_{m+1}(y_{m+1}))$$
where $s_i$ is a unary term and $y_i\in \{x_1,\ldots,x_n\}$ for every $1\leq i\leq m+1$.

Consider $a\,\alpha\,b$ and $u_i\,\beta\,v_i$ for $i=1,\dots,n$ and assume that $t^Q(a,u_1,\dots,u_n)=t^Q(a,v_1,\dots,v_n)$. The goal is to show that $t^Q(b,u_1,\dots,u_n)=t^Q(b,v_1,\dots,v_n)$. There are two cases.

\emph{Case 1:} the variable $x_0$ appears as an argument of a translation, say the $r$-th one, i.e., 
\[t(x_0,\dots,x_n)=L_{s_1(y_{1})}^{\epsilon_1}\cdots L_{s_r(x_0)}^{\epsilon_r}\cdots L_{s_m(y_{m})}^{\epsilon_m}(s_{m+1}(y_{m+1}))\]
where $y_j\in  \{x_1,\ldots,x_n\}$. We denote
\begin{align*}
g&=L_{s_1(u_{1})}^{\epsilon_1}\cdots L_{s_r(x_0)}^{\epsilon_r}\cdots L_{s_m(u_{m})}^{\epsilon_m}L_{s_m(v_{m})}^{-\epsilon_m}\cdots L_{s_r(x_0)}^{-\epsilon_r}\cdots L_{s_1(v_{1})}^{-\epsilon_1} \\ 
h&=L_{s_1(u_{1})}^{\epsilon_1}\cdots L_{s_r(y_0)}^{\epsilon_r}\cdots L_{s_m(u_{m})}^{\epsilon_m}L_{s_m(v_{m})}^{-\eps_m}\cdots L_{s_r(y_0)}^{-\eps_r}\cdots L_{s_1(v_{1})}^{-\eps_1}. 
\end{align*}
The assumption is equivalent to $g(s_{m+1}(u_{m+1}))=s_{m+1}(v_{m+1})$, the goal is equivalent to $h((s_{m+1}(u_{m+1}))=s_{m+1}(v_{m+1})$. We prove that $g=h$:
\begin{align*}
g &= L_{s_1(u_{1})}^{\eps_1}\cdots L_{s_r(x_0)}^{\eps_r}\cdots L_{s_m(u_{m})}^{\eps_m}L_{s_m(v_{m})}^{-\eps_m}\cdots L_{s_r(x_0)}^{-\eps_r}\cdots L_{s_1(v_{1})}^{-\eps_1} \\
&= L_{s_1(u_{1})}^{\eps_1}\cdots L_{s_r(x_0)}^{\eps_r}\underbrace{\cdots L_{s_m(u_{m})}^{\eps_m}L_{s_m(v_{m})}^{-\eps_m}\cdots}_{\in\dis_\beta}\underbrace{ L_{s_r(x_0)}^{-\eps_r}L_{s_r(y_0)}^{\eps_r}}_{\in\dis_\alpha}L_{s_r(y_0)}^{-\eps_r}\cdots L_{s_1(v_{1})}^{-\eps_1} \\
&= L_{s_1(u_{1})}^{\eps_1}\cdots L_{s_r(x_0)}^{\eps_r}\underbrace{L_{s_r(x_0)}^{-\eps_r}L_{s_r(y_0)}^{\eps_r}}_{\in\dis_\alpha}\underbrace{\cdots L_{s_m(u_{m})}^{\eps_m}L_{s_m(v_{m})}^{-\eps_m}\cdots}_{\in\dis_\beta}L_{s_r(y_0)}^{-\eps_r}\cdots L_{s_1(v_{1})}^{-\eps_1} = h,
\end{align*}
using $[\dis_\alpha,\dis_\beta]=1$ to commute the two underbraced mappings. In this way we can substitute all the occurrences of $x_0$ as argument of a left translation of $g$. Therefore we have $g=h$ and consequently $h(s_{m+1}(u_{m+1}))=g(s_{m+1}(u_{m+1}))=s_{m+1}(v_{m+1})$.

\emph{Case 2:} the variable $x_0$ is the rightmost variable of $t$, i.e., 
\[t(x_0,\dots,x_n)=L_{s_1(y_{1})}^{\epsilon_1}\cdots L_{s_m(y_{m})}^{\epsilon_m}(s_{m+1}(x_0))\]
where $y_{i}\in \{x_1,\ldots x_n\}$ and $s_i$ is a unary term for every $1\leq i\leq m$. We can assume that $x_0$ appear just once, otherwise we can substitute it with $y_0$ by using the argument in case 1. Let us denote again 
\[g=L_{s_1(u_{1)}}^{\epsilon_1}\cdots L_{s_m(u_{m})}^{\epsilon_m}L_{s_m(v_{m})}^{-\epsilon_m}\cdots L_{s_1(v_{1})}^{-\epsilon_1}. \] 
The assumption is equivalent to $g(s_{m+1}(a))=s_{m+1}(a)$, the goal is equivalent to $g(s_{m+1}(b))=s_{m+1}(b)$. Since $u_i\, \alpha\, v_i$ for every $i$, then according to \ref{l:words in Dis} we have that $g\in\dis_\beta$, and $s_{m+1}(a)\, \alpha\, s_{m+1}(b)$ whenever $a\, \alpha\, b$. So we can apply $\alpha$-semiregularity and the case is finished.
\end{proof}

Now, it becomes easy to characterize the commutator in ltt left quasigroups by properties of the corresponding subgroups.

\begin{proposition}\label{p:commutators}
Let $Q$ be an ltt left quasigroup and let $\alpha,\beta$ be its congruences. Then $[\alpha,\beta]$ is the smallest congruence $\delta$ such that $[\dis_{\alpha/\delta},\dis_{\beta/\delta}]=1$ and $\dis_{\beta/\delta}$ acts $\alpha/\delta$-semiregularly on $Q/\delta$.
\end{proposition}

Equivalently, we could state the former condition as $[\dis_\alpha,\dis_\beta]\leq \dis^\delta$.

\begin{proof}
Since $[\alpha,\beta]\leq\alpha\wedge\beta$, we can assume that $\delta\leq\alpha\wedge\beta$. 
Using observation (C3) for $\theta=\delta$, we obtain that $C(\alpha,\beta;\delta)$ if and only if
$[\dis_{\alpha/\delta},\dis_{\beta/\delta}]=1$ and $\dis_{\beta/\delta}$ acts $\alpha/\delta$-semiregularly on $Q/\delta$.
\end{proof}

While our characterization is not convenient to calculate the actual value of the commutator in general, it is useful in the derived concepts of abelianness and centrality, since they require the commutator to be zero. Using Proposition \ref{p:commutators} with $\beta=\alpha$ (resp. $\beta=1_Q$) and $\delta=0$, we immediately obtain one of our main results, Theorem \ref{t:abelian,central}.

\subsection{Commutator in faithful quandles}\label{s:commutators_faithful}

In general, the semiregularity condition in Proposition \ref{p:commutators} is necessary (even in the special case $\alpha=\beta=1_Q$ defining abelianness, see \cite{JPSZ2}). However, for faithful quandles, semiregularity follows from the commutativity condition, and the characterization of $C(\alpha,\beta;\delta)$ simplifies.

\begin{lemma}\label{l:faithful_semiregularity}
Let $Q$ be a faithful quandle, $\alpha$ its congruence and $N\leq\dis(Q)$. If $[N,\dis_\alpha]=1$, then $N$ acts $\alpha$-semiregularly on $Q$.
\end{lemma}

\begin{proof}
Consider $h\in N$ and $a\in Q$ such that $h(a)=a$. Take any $b\,\alpha\, a$. Then
\begin{displaymath}
L_{b}=L_{b} L_{a}^{-1} L_{a}=L_{b} L_{a}^{-1} L_{h(a)}=\underbrace{ L_{b} L_{a}^{-1}}_{\in \dis_\alpha} \underbrace{h}_{\in N} L_{a} h^{-1}=\underbrace{h}_{\in N}\underbrace{ L_{b} L_{a}^{-1}}_{\in \dis_\alpha}L_{a} h^{-1}=h  L_{b} h^{-1}=L_{h(b)}.
\end{displaymath}
Since $Q$ is faithful, we obtain that $h(b)=b$.
\end{proof}

\begin{corollary}\label{abelianness faithful case}
Let $Q$ be a faithful quandle and $\alpha$ its congruence. Then
\begin{itemize}
\item[(1)] $\alpha$ is abelian if and only if $\dis_\alpha$ is abelian.
\item[(2)] $\alpha$ is central if and only if $\dis_\alpha$ is central in $\dis(Q)$.
\end{itemize}
\end{corollary}

For general commutator, we must assume that every factor is faithful, so that we can properly use Corollary \ref{abelianness faithful case}. This assumption holds, for example, in every finite latin quandle, or in every quandle that has CDSg.

\begin{proposition}\label{p:comm_comm}
Let $Q$ be a quandle such that every factor of $Q$ is faithful, and let $\alpha,\beta$ be its congruences. Then \[ [\alpha,\beta]=[\beta,\alpha]=\con{[\dis_{\alpha},\dis_\beta]}=\mathcal{O}_{[\dis_\alpha,\dis_\beta]}. \]
\end{proposition}

\begin{proof}
Proposition \ref{p:commutators} and Lemma \ref{l:faithful_semiregularity} apply to every triple of congruences, since every factor of $Q$ is faithful. Hence $[\alpha,\beta]$ is the smallest congruence $\delta$ such that $[\dis_\alpha,\dis_\beta]\leq \dis^\delta$ and we see that the commutator is commutative.
 
Let $\omega=\con{[\dis_{\alpha},\dis_\beta]}$. 
Proposition \ref{p:con_N}(1) for $N=[\dis_\alpha,\dis_\beta]$ says that $[\dis_\alpha,\dis_\beta]\leq \dis^\omega$. 
Now consider any congruence $\delta$ satisfying $[\dis_\alpha,\dis_\beta]\leq \dis^\delta$. 
Since the operator $\con{}$ is monotone, we have $\omega\leq\con{\dis^\delta}=\delta$ by Proposition \ref{p:con_dis}.
Hence $[\alpha,\beta]=\omega$.

Let $\gamma=\mathcal{O}_{N}$. Then according to Lemma \ref{l:kernel} then $N\leq \dis^\gamma$ and so $C(\alpha,\beta,\gamma)$ holds. Then $[\alpha,\beta]=\con{N}\leq \gamma$. Hence by Proposition \ref{p:commutators}(4), the equality holds.
\end{proof}

In general, the commutator is not commutative, not even in the special case $[\alpha,1_Q]$. Indeed, the Cayley kernel causes troubles.

\begin{example}\label{ex:non-commutative commutator}
On one hand, $[\alpha,\lambda_Q]=0_Q$ in every quandle and for every $\alpha$, since $\dis_{\lambda_Q}=1$ and it satisfies the conditions of Lemma \ref{l:C(ab0)} trivially. In particular $[1_Q,\lambda_Q]=0_Q$. 
There are examples of non-faithful connected quandles where $[\lambda_Q,1_Q]\neq0_Q$, that is, where $\dis(Q)$ does not act $\lambda_Q$-semiregularly, such as {\tt SmallQuandle}(30,4), (36,58) and (45,29) in the RIG library.
\end{example}

\subsection{The center of a rack}

We will use Theorem \ref{t:abelian,central} to calculate the center of a rack. One natural property to expect is that every central pair \emph{mediates} with all other pairs, i.e., if $a\,\zeta_Q\,b$ then $(u*a)*(b*v)=(u*b)*(a*v)$ for every $u,v$. (A rack is medial if and only if all pairs mutually mediate.

\begin{lemma}\label{l:mediators}
Let $Q$ be a rack. Then \[ \con{Z(\dis(Q))}=\{(a,b):\ (u*a)*(b*v)=(u*b)*(a*v)\text{ for every }u,v\in Q\}.\]
\end{lemma}

\begin{proof}
Observe that $(u*a)*(b*v)=(u*b)*(a*v)$ for every $u,v$ if and only if $L_{u*a}L_b=L_{u*b}L_a$ for every $u$, which is equivalent to (\dag) $L_aL_u^{-1}L_b=L_bL_u^{-1}L_a$ using equation \eqref{L_f(a)}. 

$(\subseteq)$ If $a\,\con{Z(\dis(Q))}\,b$, then $L_bL_a^{-1}\in Z(\dis(Q))$, and thus $L_aL_u^{-1}L_bL_a^{-1}=L_bL_a^{-1}L_aL_u^{-1}=L_bL_u^{-1}$ for every $u$, and (\dag) follows immediately.

$(\supseteq)$ First note that (\dag) is equivalent to $L_u^{-1}L_aL_b^{-1}=L_b^{-1}L_aL_u^{-1}$. Now, use this identity and its inverse to conclude that $L_uL_v^{-1}L_aL_b^{-1}=L_uL_b^{-1}L_aL_v^{-1}=L_aL_b^{-1}L_uL_v^{-1}$
for every $u,v$.
\end{proof}

To handle the semiregularity condition, we define an equivalence $\sigma_Q$ on $Q$ by
\[ a \,\sigma_Q\, b \ \Leftrightarrow \ \dis(Q)_a=\dis(Q)_b. \]
Here $\dis(Q)_x$ denotes the stabilizer of $x$ in $\dis(Q)$.

\begin{lemma}\label{l:sigma}
Let $Q$ be a rack. Then $\dis(Q)$ acts $\alpha$-semiregularly on $Q$ if and only if $\alpha\leq\sigma_Q$.
\end{lemma}

\begin{proof}
By definition, $\dis(Q)$ acts $\alpha$-semiregularly on $Q$ if and only if
$f(a)=a\Leftrightarrow f(b)=b$ for every $f\in\dis(Q)$ and every $a\,\alpha\,b$. This is equivalent to saying that the stabilizers $\dis(Q)_a$ and $\dis(Q)_b$ coincide for every $a\,\alpha\,b$, which is equivalent to $\alpha\leq\sigma_Q$.
\end{proof}

\begin{proposition}\label{p:center}
Let $Q$ be a rack. Then $\zeta_Q=\con{Z(\dis(Q))}\cap\sigma_Q$.
\end{proposition}

\begin{proof}
Let $\xi=\con{Z(\dis(Q))}\cap\sigma_Q$. First, we prove that $\xi$ is a congruence of $Q$. It is an intersection of two equivalences, hence it is an equivalence. Let $a\,\xi\,b$ and $c\in Q$. Since $\con{Z(\dis(Q))}$ is a congruence, it remains to prove that $(c\ast a)\,\sigma_Q\,(c\ast b)$ and $(a\ast c)\,\sigma_Q\,(b\ast c)$ (for left division, the proof is analogical). For the former claim, assume that $f(c\ast a)=c\ast a$, $f\in\dis(Q)$, or equivalently, $L_c^{-1}fL_c(a)=a$. Since $L_c^{-1}fL_c\in\dis(Q)$ and $a\,\sigma_Q\,b$, we have $L_c^{-1}fL_c(b)=b$, and thus $f(c\ast b)=c\ast b$. For the latter claim, assume that $f(a\ast c)=a\ast c$,  $f\in\dis(Q)$, or equivalently, $L_a^{-1}fL_a(c)=c$. Then 
\[ L_b^{-1}fL_b(c)=L_b^{-1}f\underbrace{L_bL_a^{-1}}_{\in Z(\dis(Q))}L_a(c)=L_b^{-1}\underbrace{L_bL_a^{-1}}_{\in Z(\dis(Q))}fL_a(c)=L_a^{-1}fL_a(c)=c,\]
and thus $f(b\ast c)=b\ast c$. We used the assumption that $a\,\con{Z(\dis(Q))}\,b$, which means that $L_bL_a^{-1}\in Z(\dis(Q))$.

In the next step, we show that every central congruence $\alpha$ is contained in $\xi$. Indeed, by Theorem \ref{t:abelian,central}, $\dis_\alpha\leq Z(\dis(Q))$, hence $\alpha\leq\con{\dis_\alpha}\leq\con{Z(\dis(Q))}$ using Proposition \ref{galois_connection} (the first inequality follows from the closure property, the second inequality from monotonicity). Lemma \ref{l:sigma} assures that $\alpha\leq\sigma_Q$.

Finally, we verify that $\xi$ is a central congruence. To show that $\dis_\xi$ is central in $\dis(Q)$, it is sufficient to look at the generators $L_aL_b^{-1}$, $a\,\xi\, b$. Then $a\,\con{Z(\dis(Q))}\,b$, hence $L_aL_b^{-1}\in Z(\dis(Q))$. Since $\xi\leq\sigma_Q$, Lemma \ref{l:sigma} assures that $\dis(Q)$ acts $\xi$-regularly.
\end{proof}

\begin{corollary}\label{c:center_faith}
Let $Q$ be a faithful quandle. Then $\zeta_Q=\con{Z(\dis(Q))}$.
\end{corollary}

\begin{proof}
Let $Q$ be faithful. We prove that $\con{Z(\dis(Q))}\leq\sigma_Q$. By Lemma \ref{l:sigma}, we shall prove that $\dis(Q)$ acts $\con{Z(\dis(Q))}$-semiregularly on $Q$. We have
\[ [\dis(Q),\dis_{\con{Z(\dis(Q))}}]\leq[\dis(Q),Z(\dis(Q))]=1, \]
and thus semiregularity follows from Lemma \ref{l:faithful_semiregularity}.
\end{proof}

\begin{corollary}\label{c:center_med}
Let $Q$ be a medial rack. Then $\zeta_Q=\sigma_Q$.
\end{corollary}

\begin{proof}
If $Q$ is medial, then $\dis(Q)$ is abelian, and thus $\con{Z(\dis(Q))}=1_Q$.
\end{proof}

\subsection{The $\O_N$ and $\lambda_Q$ congruences}\label{ss:special_con}

\begin{lemma}\label{abelian subgroup gives abelian cong}
Let $Q$ be a rack and $N\in\N(Q)$ abelian (resp. central in $\dis(Q)$). Then $\O_N$ is an abelian (resp. central) congruence of $Q$.  In particular $Z(\dis(Q))\leq \dis^{\zeta_Q}$.
\end{lemma}

\begin{proof}
According to Theorem \ref{t:abelian,central}, we need to check that $\dis_{\O_N}$ is abelian (resp. central in $\dis(Q)$) and that $\dis_{\O_N}$ (resp. $\dis(Q)$) acts $\O_N$-semiregularly on $Q$.

First, observe that $\dis_{\O_N}\leq N$. By Proposition \ref{p:con_N}(4), $\O_N\leq\con N$, and applying the Galois connection we obtain $\dis_{\O_N}\leq\dis_{\con N}\leq N$.

Consequently, $\dis_{\O_N}$ is abelian (resp. central in $\dis(Q)$), since so is $N$. Let $f\in\dis_{\O_N}$ (resp. $f\in\dis(Q)$) and consider $a\in Q$ such that $f(a)=a$. For any $b\,\O_N\,a$, take $g\in N$ such that $b=g(a)$. Then $f(b)=f(g(a))=g(f(a))=g(a)=b$, where $fg=gf$ follows from abelianness (resp. centrality) of $N$.

In particular $\mathcal{O}_{Z(\dis(Q))}\leq \zeta_Q$, so using Lemma \ref{l:kernel} $Z(\dis(Q))\leq \dis^{\zeta_Q}$.
\end{proof}

\begin{proposition}\label{p:medial=>nilpotent}
Medial racks are nilpotent of length at most 2.
\end{proposition}

\begin{proof}
If $Q$ is a medial rack, then $\dis(Q)$ is abelian, and the chain $0_Q\leq\O_{Q}\leq1_Q$ is a witness.
The congruence $\O_{Q}$ is central by Lemma \ref{abelian subgroup gives abelian cong} for $N=\dis(Q)$. The factor $Q/\O_{Q}$ is a permutation rack, and thus abelian: for $[a],[b],[c]\in Q/\O_{Q}$, we have $[a]*[c]=[a*c]=[b*c]=[b]*[c]$, because $b*c=L_bL_a^{-1}(a*c)$.
\end{proof}

The Cayley kernel is always abelian: indeed, $\dis_{\lambda_Q}=1$, hence it is abelian and acts $\lambda_Q$-semiregularly.
However, it may not be central.

\begin{example}
Consider the quandle $Q$ with the multiplication table
\begin{center}
\begin{tabular}{|cccc| }
\hline
1 & 2 & 3 & 4 \\ 
1 & 2 & 4 & 3 \\
1 & 2 & 3 & 4 \\
1 & 2 & 3 & 4  \\
\hline
\end{tabular}\,.
\end{center}
Indeed, $\dis(Q)=\langle f\rangle$ where $f=(3\ 4)$, but it does not act $\lambda_Q$-semiregularly, since $f(1)=1$, $1\,\lambda_Q\,3$ and $f(3)=4$.
\end{example}

\section{Nilpotent and solvable racks}\label{s:6}

\subsection{Nilpotence and solvability of racks, and of their associated groups}\label{s:nilpotent,solvable}

The two lemmas below prove our second main result, Theorem \ref{t:nilpotent,solvable}.
In the two proofs, let $\Gamma^{(n)}$ denote the $n$-th member of the derived series, and $\Gamma_{(n)}$ the $n$-th subgroup of the lower central series, of a given group. The subgroups $\Gamma^{(n)}$ and $\Gamma_{(n)}$ correspond to the group congruences $\gamma^{n}$ and $\gamma_{n}$, respectively. 

\begin{lemma}\label{solvable then dis solv}
Let $Q$ be a rack. If $Q$ is nilpotent (resp. solvable) of length $n$, then $\dis(Q)$ is a nilpotent (resp. solvable) group of length $\leq 2n-1$.
\end{lemma}

\begin{proof}
We proceed by induction on the length $n$. For $n=1$, the rack $Q$ is abelian, hence $\dis(Q)$ is an abelian group and the statement holds.
In the induction step, assume that the statement holds for all racks that are nilpotent (resp. solvable) of length $\leq n-1$.
Consider a chain of congruences \[ 0_{Q}=\alpha_0\leq\alpha_1\leq\ldots\leq \alpha_n= 1_{Q}\]
such that $\alpha_{i+1}/\alpha_i$ is central (resp. abelian) in $Q/\alpha_i$, for every $i$. In particular, $\alpha_1$ is central (resp. abelian) in $Q$ and the rack $Q/\alpha_1$ is nilpotent (resp. solvable) of length $n-1$, as witnessed by the series 
\[ 0_{Q/\alpha_1}=\alpha_1/\alpha_1\leq\alpha_2/\alpha_1\leq\ldots\leq \alpha_n/\alpha_1=1_{Q/\alpha_1} \]
(see Lemma \ref{l:2rd_iso_thm}). By the induction assumption, $\dis(Q/\alpha_1)$ is nilpotent (resp. solvable) of length $\leq2n-3$.
Now, consider the series $\Gamma_{(i)}$ (resp. $\Gamma^{(i)}$) in $\dis(Q)$ and project it into $\dis(Q/\alpha_1)$. Since $\pi_{\alpha_1}(\Gamma_{(2n-3)})=1$, we obtain that
$\Gamma_{(2n-3)}\leq\Ker{\pi_{\alpha_1}}=\dis^{\alpha_1}$ (resp. analogically for $\Gamma^{(2n-3)}$). Now, in case of nilpotence, we have
\begin{align*}
\Gamma_{(2n-1)} &= [[\Gamma_{(2n-3)},\dis(Q)],\dis(Q)] \\
&\leq [[ \dis^{\alpha_1},\dis(Q)],\dis(Q)] \leq [\dis_{\alpha_1},\dis(Q)]=1,
\end{align*}
using Proposition \ref{p:dis_alpha2}(1) in the penultimate step, and centrality of $\dis_{\alpha_1}$ (by Theorem \ref{t:abelian,central}) in the ultimate step. In case of solvability, we have
\begin{align*}
\Gamma^{(2n-1)} &= [[\Gamma^{(2n-3)},\Gamma^{(2n-3)}],[\Gamma^{(2n-3)},\Gamma^{(2n-3)}]] \\
&\leq [[ \dis^{\alpha_1},\dis^{\alpha_1}],[ \dis^{\alpha_1},\dis^{\alpha_1}]] \leq [\dis_{\alpha_1},\dis_{\alpha_1}]=1, 
\end{align*}
using Proposition \ref{p:dis_alpha2}(1), and abelianness of $\dis_{\alpha_1}$.
\end{proof}

The bound on the length is tight already for $n=2$. For example, one can check in the RIG library that non-principal latin quandles of size 27 are nilpotent of length $2$, but their displacement groups are nilpotent of length $3$; and non-principal latin quandles of size 28 are solvable of length $2$, but their displacement groups are solvable of length $3$.

\begin{lemma}\label{Dis solv implies solv}
Let $Q$ be a rack. If $\dis(Q)$ is nilpotent (resp. solvable) of length $n$, then $Q$ is nilpotent (resp. solvable) of length $\leq n+1$.
\end{lemma}

\begin{proof}
We proceed by induction. 
For $n=1$, the group $\dis(Q)$ is abelian, and Proposition \ref{p:medial=>nilpotent} assures that $Q$ is nilpotent (and thus solvable, too) of length $\leq2$.
In the induction step, assume that the statement holds for all racks with the displacement group nilpotent (resp. solvable) of length $\leq n-1$.

First, observe that, for any $N\in\N(Q)$, $\pi_{\O_N}(N)=1$: indeed, for every $f\in N$ and $a\in Q$, we have $f(a)\,\O_N\,a$, hence $f$ acts identically on $Q/\O_N$. Therefore, $N\leq\Ker{\pi_{\O_N}}=\dis^{\O_N}$.

Now, take $N=\Gamma_{(n-1)}$ (resp. $N=\Gamma^{(n-1)}$). Since $N$ is central (resp. abelian) in $\dis(Q)$, the congruence $\O_N$ is also central (resp. abelian), by Lemma \ref{abelian subgroup gives abelian cong}. From the observation we obtain that $N\leq\dis^{\O_N}$. Therefore, the group $\dis(Q/\O_N)\simeq\dis(Q)/\dis^{\O_N}$ is nilpotent (resp. solvable) of length $\leq n-1$, and by the induction assumption, $Q/\O_N$ is nilpotent (resp. solvable) of length $m\leq n$.
Let
\[ 0_{Q/\O_N}=\O_N/\O_N\leq\alpha_1/\O_N\leq\ldots\leq \alpha_m/\O_N=1_{Q/\O_N} \]
be the chain of congruences that witnesses nilpotence (resp. solvability). Then the chain
\[ 0_Q\leq\O_N\leq\alpha_1\leq\ldots\leq \alpha_m=1_Q \]
witnesses that $Q$ is nilpotent (resp. solvable) of length $\leq n+1$, using Lemma \ref{l:2rd_iso_thm}.
\end{proof}

The bound on the length is tight already for $n=1$. For example, the 3-element quandle with two orbits has an abelian displacement group, but the its action is not semiregular, hence $Q$ cannot be abelian \cite{JPSZ2}.

According to \cite{Smith}, finite (left and right) distributive quasigroups have nilpotent displacement groups, so we have the following corollary.

\begin{corollary}\label{finite distributive}
Finite distributive quasigroups are nilpotent.
\end{corollary}

\subsection{Prime decomposition for nilpotent quandles}\label{s:nilpotent_prime_decomposition}

Theorem \ref{t:nilpotent,solvable} allows to transfer certain properties from groups to racks.

\begin{proposition}\label{p:HSP}
Every subquandle and every quotient of a nilpotent (resp. solvable) quandle is nilpotent (resp. solvable). 
The direct product of finitely many nilpotent (resp. solvable) quandles is nilpotent (resp. solvable).
\end{proposition}

\begin{proof}
Let $Q$ be a nilpotent (resp. solvable) quandle. By Theorem \ref{t:nilpotent,solvable}, the group $\dis(Q)$ is nilpotent (resp. solvable).

Consider a subquandle $S$ of $Q$, and let $H=\langle L_aL_b^{-1}:\ a,b\in S\rangle\leq\dis(Q)$. Then $H$ is nilpotent (resp. solvable), and $H\to\dis(S)$, $h\mapsto h|_S$, is a surjective group homomorphism. Hence $\dis(S)$ is nilpotent (resp. solvable), and so is $S$ by Theorem \ref{t:nilpotent,solvable}. 

Consider a congruence $\alpha$ of $Q$. Then $\pi_\alpha:\dis(Q)\to\dis(Q/\alpha)$ is a surjective group homomorphism. Hence $\dis(Q/\alpha)$ is nilpotent (resp. solvable), and so is $Q/\alpha$ by Theorem \ref{t:nilpotent,solvable}.

Let $\setof{Q_i}{1\leq i\leq n}$ be a set of quandles, $Q=\prod_{i=1}^n Q_i$ and let $\alpha_i$ be the kernel of the canonical mapping $Q\longrightarrow Q_i$. Then $\bigcap_{i=1}^n\dis^{\alpha_i}=1$ and so the group homomorphism
$$\dis(Q)\longrightarrow \prod_{i=1}^n \dis(Q_i), \quad h\mapsto (\pi_{\alpha_1}(h),\ldots , \pi_{\alpha_n}(h))$$
is injective. Using again Theorem \ref{t:nilpotent,solvable}, Iif the quandles $\setof{Q_i}{1\leq i\leq n}$ are nilpotent (resp. solvable), then so it is $\prod_{i=1}^n \dis(Q_i)$. Hence $\dis(Q)$ is nilpotent (resp. solvable) and so it is $Q$.
\end{proof}

The following fact was proved in Bianco's PhD thesis \cite{Bianco}. For reader's convenience, we include his proof.

\begin{proposition}\cite[Corollary 5.2]{Bianco} \label{p:p-group}
Let $Q$ be a connected rack of prime power size $p^k$. Then $\dis(Q)$ is a $p$-group. 
\end{proposition}

\begin{proof}
According to \cite[Theorem A.2]{EGS}, if $G$ is a finite group, $C$ a conjugacy class of prime power size $p^n$ and $G=\langle C\rangle$, then $G/O_p(G)$ is cyclic, where $O_p(G)$ denotes the $p$-core of $G$. 
In particular, set $G=\lmlt(Q)$ and $C=\{L_a:a\in Q\}$. By \cite[Lemma 1.29]{AG}, $|C|$ divides $|Q|$, hence it is a prime power, and thus $G/O_p(G)$ is cyclic. Therefore, $G'\leq O_p(G)$ is a $p$-group. According to \cite[Proposition 3.2]{HSV}, $G'=\dis(Q)$ for connected quandles. 
\end{proof}

An immediate consequence of Proposition \ref{p:p-group} and Theorem \ref{t:nilpotent,solvable} is that connected racks of prime power size are nilpotent. Now, we are ready to prove the third main result, Theorem \ref{t:nilpotent_prime_decomposition}, about the primary decomposition of nilpotent quandles.

\begin{proof}[Proof of Theorem \ref{t:nilpotent_prime_decomposition}]
$(\Leftarrow)$ Connected quandles of prime power size are nilpotent, and their product is nilpotent, too, by Proposition \ref{p:HSP}.

$(\Rightarrow)$
The proof is based on the minimal representation of connected quandles, see \cite[Section 4]{HSV} for all undefined notions.
According to \cite[Theorem 7.1]{J} or \cite[Proposition 3.5]{HSV}, every connected quandle $Q$ is isomorphic to the coset quandle $\Q_{\mathrm{Hom}}(\dis(Q),\dis(Q)_e,\widehat L_e)$ where $e\in Q$ is chosen arbitrarily and $\widehat L_e$ denotes the inner automorphism given by $L_e$. Since $Q$ is faithful, $\dis(Q)_e=\mathrm{Fix}(\widehat{L}_e)$, the set of fixed points of $\widehat L_e$: indeed, for $g\in\dis(Q)$, we have $L_eg=gL_e$ if and only if $e*g(a)=g(e*a)=g(e)*g(a)$ for every $a\in Q$, which is equivalent to $L_e=L_{g(e)}$, which in turn is equivalent to $g(e)=e$ by faithfulness. 

Now, if $Q$ is nilpotent, then so is $\dis(Q)$, and it decomposes as the direct product of its Sylow subgroups, $\dis(Q)\simeq\prod S_p$, where $S_p$ is the $p$-Sylow subgroup of $\dis(Q)$. Sylow subgroups are invariant under automorphisms of $\dis(Q)$, therefore $\widehat L_e$ decomposes as the product of the restrictions of $\widehat{L_e}$ to the Sylow subgroup, and the subgroup $\mathrm{Fix}(\widehat L_e)$ decomposes, too. Therefore, 
\[ Q\simeq \mathcal{Q}_{\mathrm{Hom}}(\prod S_p,\prod \mathrm{Fix}(\widehat{L}_e|_{S_p}), \prod \widehat{L}_e|_{S_p}) \simeq \prod \mathcal{Q}_{\mathrm{Hom}}( S_p, \mathrm{Fix}(\widehat{L}_e|_{S_p}), \widehat{L}_e|_{S_p}), \]
which is a product of connected quandles of prime power size.
\end{proof}

\section{Abelian and central extensions}\label{s:7}

\subsection{Constructing extensions}

Let $Q$ be a left quasigroup, $A$ an abelian group and $\phi,\psi,\theta$ mappings
\begin{eqnarray}\label{eq:ext}
\phi: Q\times Q \to \End A, \quad \psi: Q\times Q \to  \Aut A,\quad \theta: Q\times Q \to A.
\end{eqnarray}
Define an operation on the set $Q\times A$ by
\[ (a,s)*(b,t)=(a*b,\phi_{a,b}(s)+\psi_{a,b}(t)+\theta_{a,b}), \]
for every $a,b\in Q$ and $s,t\in A$. The resulting left quasigroup
\[ Q\times_{\phi,\psi,\theta} A=(Q\times A,*,\ld) \]
will be called an \emph{abelian extension} of $Q$ by the triple $(\phi,\psi,\theta)$.
If $\phi,\psi$ are constant mappings, we will call it a \emph{central extension}.
The mapping $Q\times_{\phi,\psi,\theta} A \to Q$, $(a,s)\mapsto a$, is a homomorphism, called \emph{canonical projection}.
(Our terminology is justified by Propositions \ref{p:abelian_ext}, \ref{p:central_ext_rep} and Remark \ref{r:ext_ua}.)

\begin{lemma}\label{l:ext}
Let $Q$ be a rack, $A$ an abelian group and $\psi,\phi,\theta$ as in \eqref{eq:ext}. Then the abelian extension $E=Q\times_{\phi,\psi,\theta} A$ is a rack if and only if
\begin{align}
\label{eq:ext1} \psi_{a,b*c}(\theta_{b,c})+\theta_{a,b*c} & = \psi_{a*b,a*c}(\theta_{a,c})+\phi_{a*b,a*c}(\theta_{a,b})+\theta_{a*b,a*c},\\
\label{eq:ext2} \psi_{a,b*c}\psi_{b,c} & = \psi_{a*b,a*c}\psi_{a,c},\\ 
\label{eq:ext3} \psi_{a,b*c}\phi_{b,c} & =  \phi_{a*b,a*c}\psi_{a,b},\\ 
\label{eq:ext4} \phi_{a,b*c} &= \phi_{a*b,a*c}\phi_{a,b}+\psi_{a*b,a*c}\phi_{a,c}
\end{align}
for every $a,b,c\in Q$. The extension $E$ is a quandle if and only if, additionally, $Q$ is a quandle and
\begin{eqnarray}\label{eq:ext_qua}
\theta_{a,a}=0 \quad\text{ and }\quad \phi_{a,a}+\psi_{a,a}=1.
\end{eqnarray}
\end{lemma}

\begin{proof}
Straightforward computation (see also \cite[Section 2]{AG}).
\end{proof}

The concept of abelian extensions for racks appeared in \cite[Section 2.3]{AG} in the following terminology. A triple $(A,\phi,\psi)$ satisfying \eqref{eq:ext2}, \eqref{eq:ext3}, \eqref{eq:ext4} is called a \emph{$Q$-module}; it is called a \emph{quandle $Q$-module} if it also satisfies \eqref{eq:ext_qua}. Then, a mapping $\theta$ satisfying \eqref{eq:ext1} is called a \emph{2-cocycle} over the $Q$-module, and the extension $Q\times_{\phi,\psi,\theta} A$ is called an \emph{affine module over $Q$}. The expression 
\[ \beta_{a,b}(s,t)=\phi_{a,b}(s)+\psi_{a,b}(t)+\theta_{a,b} \]
defines a \emph{dynamical cocycle}. Therefore, the concept of abelian extensions is a special case of the general concept of \emph{extensions by dynamical cocycles}, denoted $Q\times_\beta A$, where the operation on $Q\times A$ is defined by $(a,s)*(b,t)=(a*b,\beta_{a,b}(s,t))$ for some dynamical cocycle $\beta$. Extensions by dynamical cocycles capture precisely homomorphisms with uniform kernels \cite[Corollary 2.5]{AG}. 

\begin{example}
A surjective homomorphism $E\to Q$ whose kernel is contained in $\lambda_E$ is often called shortly a \emph{covering} of $Q$ \cite{Eis}. As noted in Section \ref{ss:special_con}, the kernel of a covering is always abelian, but not necessarily central. Uniform coverings can be represented as \emph{extensions by constant cocycles} \cite[Proposition 2.11]{AG}. Central extensions with $\varphi_{a,b}=0$ and $\psi_{a,b}=1$ are special cases that were studied extensively in \cite{CS,CSV} under the name ``abelian extensions". The relation of various types of coverings to general universal algebraic concepts is the topic of our subsequent paper \cite{BS-cov}.
\end{example}

\begin{example}
The \emph{semiregular extensions} $\mathrm{Ext}(A,f,d_a:a\in A)$ from \cite{JPSZ2} are special cases of central extensions where $Q$ is a projection quandle, $\phi_{a,b}=1-f$, $\psi_{a,b}=f$ and $\theta_{a,b}=d_a-d_b$. Semiregular extensions are proved to represent abelian quandles.
\end{example}

\begin{example}
\emph{Galkin quandles} studied in \cite{CEHSY,CH} are special cases of abelian extensions where $Q=\Aff(\Z_3,-1)$, $A$ is an arbitrary abelian group, $u\in A$ and
\[ \phi_{a,b}= \begin{cases} 2 & a=b \\ -1 & a\neq b \end{cases}, \quad \psi_{a,b}=-1,\quad \theta_{a,b}=\begin{cases} u & a+2=b \\ 0 & a+2\neq b \end{cases}. \]
\end{example}

In \cite{AG,CEGS,Jack}, abelian extensions are studied from the viewpoint of cohomology theory. In \cite{CEGS}, they are used to construct knot invariants. Here we ask a different question: which surjective rack homomorphisms can be represented by abelian and central extensions? We start with an observation that the kernel of the canonical projection of an abelian (resp. central) extension is an abelian (resp. central) congruence. 

\begin{proposition}\label{p:abelian_ext}
Let $E=Q\times_{\phi,\psi,\theta} A$ be an abelian (resp. central) extension of a rack $Q$. Then the kernel of the canonical projection $E\to Q$ is an abelian (resp. central) congruence.
\end{proposition}

\begin{proof}
The kernel congruence $\alpha$ is defined by $(a,s) \ \alpha \ (b,t)$ iff $a=b$.
Using Theorem \ref{t:abelian,central}, it is sufficient to prove that $\dis_\alpha$ is abelian (resp. central) and acts $\alpha$-semiregularly on $E$ (resp. $\dis(E)$ does).
It is straightforward to calculate that 
\[ L_{(a,s)}L_{(a,t)}^{-1}(c,r)=(c,\phi_{a,a\ld c}(s-t)+r).\]
We see that any two displacements $L_{(a,s)}L_{(a,t)}^{-1}$ and $L_{(b,r)}L_{(b,u)}^{-1}$ commute, and thus $\dis_\alpha$ is an abelian group. 
Let $h\in \dis_\alpha$. Then $h(c,r)=(c,x_{h,c}+r)$ where $x_{h,c}\in A$ is an element which only depends on $h$ and $c$. Hence, if $h(c,r)=(c,r)$, then $x_{h,c}=0$, and thus $h(c,s)=(c,s)$ for every $s\in A$. Therefore, $\dis_\alpha$ acts $\alpha$-semiregularly on $E$. 

If $\phi$ and $\psi$ are constant mappings, we have  
\begin{gather*}
L_{(a,s)}L_{(a,t)}^{-1}(d,r)=(d,\phi(s-t)+r),\\
L_{(b,u)}L_{(c,v)}^{-1}(d,r)=(b\cdot (c\backslash d),\phi(u-v)+\theta_{b,c\backslash d}-\theta_{c,c\backslash d}+r),
\end{gather*}
We see that these mappings commute, hence $\dis_\alpha$ is central in $\dis(Q)$. A similar argument shows that $\dis(E)$ atcs $\alpha$-semiregularly on $E$.
\end{proof}

\subsection{Representing by extensions}

Consider a surjective quandle homomorphism $f:E\to Q$. Equivalently, consider a quandle $E$ and its congruence $\alpha$, and put $Q=E/\alpha$.
We will say that $f$ (resp. $\alpha$) \emph{admits a representation} by an abelian or central extension if $E\simeq Q\times_{\phi,\psi,\theta}A$ for suitable $A,\phi,\psi,\theta$.
Under which conditions do we obtain such a representation? 

Indeed, the blocks of the kernel of $f$ (resp. $\alpha$ itself) must have equal size, i.e., the congruence must be uniform (cf. Proposition \ref{p:uniform}).
Another natural constraint was given in Proposition \ref{p:abelian_ext}. 

But there are more. One sort of troubles comes from the following fact. The blocks of an abelian extension are subquandles which are affine, all of them over the same group. The blocks of a uniform abelian congruence are subquandles which are abelian. However, this is a weaker condition: according to \cite[Theorem 2.2]{JPSZ2}, abelian quandles embed into affine quandles, but they are not necessarily affine. And even if the blocks were affine, then not necessarily over the same abelian group. 

\begin{example}
Let $R$ be an abelian quandle which is not affine (e.g., the three-element quandle with two orbits). Then the congruence $1_R$ is uniform and abelian, but it does not admit a representation by an abelian extension. This can be turned into a proper uniform congruence by taking any direct product $Q\times R$.
\end{example}

\begin{example}
Let $(Q_i,*_i,\ld_i)$, $i\in I$, be quandles with disjoint underlying sets. Put $Q=\bigcup Q_i$ and define an operation on $Q$ by
$a*b=a*_i b$ if both $a,b\in Q_i$, and $a*b=b$ otherwise; similarly for left division. It is straightforward to check that $(Q,*,\ld)$ is a quandle, and if all $Q_i$ are connected, they form the orbits of $Q$. Now, assume that all $Q_i$ are affine. Then the congruence $\O_Q$ is abelian, since $\dis_{\O_Q}$ is abelian and acts $\O_Q$-semiregularly on $Q$. But if they are affine over different groups, the congruence $\O_Q$ does not admit a representation by an abelian extension.
\end{example}

The former problem can be avoided by the assumption that the blocks are connected, since connected abelian quandles are affine by \cite[Theorem 7.3]{HSV}. For the latter problem, we need some homogenity assumption. One natural condition is that $Q/\alpha$ is connected, which forces the blocks to be isomorphic by Proposition \ref{p:uniform}.

Finite latin quandles satisfy both assumptions for every congruence. Here are other examples: the RIG library contains non-latin quandles where all subquandles are connected, {\tt SmallQuandle}(28,$k$) for $k=3,4,5,6$.

Now we prove a representation result for central congruences. Somewhat weaker assumptions are actually sufficent.

\begin{proposition}\label{p:central_ext_rep} 
Let $E$ be a quandle, and $N\in\N(E)$ central in $\dis(E)$ such that $E/\mathcal{O}_N$ is connected. Then $E$ is isomorphic to a central extension $E/\mathcal{O}_N\times_{\phi,\psi,\theta} A$ for some $\phi,\psi,\theta$.
\end{proposition}

\begin{proof}
Denote $\alpha=\mathcal{O}_N$.
According to Lemma \ref{abelian subgroup gives abelian cong}, $\alpha$ is a central congruence, and so $\dis_\alpha$ is central and $\dis(E)$ atcs $\alpha$-semiregularly on $E$ by Theorem \ref{t:abelian,central}.

Pick $e\in E$. We define group operations on $[e]$ in the following way: for $a,b\in[e]$, take any $f\in \dis(Q)$ such that $f(e)=a$, and let $a+b=f(b)$ and $-a=f^{-1}(e)$. The operations are well defined, since $\dis(E)$ acts $\alpha$-semiregularly: hence, if $f_1(e)=a=f_2(e)$, then $f_1$ and $f_2$ coincide on $[e]$, and thus $f_1(b)=f_2(b)$ and $f_1^{-1}(e)=f_2^{-1}(e)$. 

First, observe that $A=([e],+,-,e)$ is an abelian group. For associativity, take $a=f(e)$, $b=g(e)$ and $c=h(e)$ with $f,g,h\in\dis(E)$ and calculate 
\[ a+(b+c)=f(b+c)=f(g(c))=f(g(h(e)))=f(g(e))+h(e)=f(b)+c=(a+b)+c.\]
For commutativity, take $a=f(e)$ and $b=g(e)$ with $f,g\in\dis_\alpha$, and use commutativity of $\dis_\alpha$ to calculate 
\[ a+b=f(b)=f(g(e))=g(f(e))=g(a)=b+a.\] 
Next, observe that $L_e$ is its automorphism of $A$: for $a=f(e)$ and $b=g(e)$, we calculate
\[ L_e(a)+L_e(b)=L_efL_e^{-1}(e)+L_egL_e^{-1}(e)=L_efL_e^{-1}L_egL_e^{-1}(e)=L_e(fg(e))=L_e(a+b).\]
Finally, the subquandle $[e]$ is equal to $\aff(A,L_e)$: for $a=f(e)$ and $b=g(e)$ with $f,g\in N$, we calculate 
\[ a-L_e(a)+L_e(b)=f(e)-L_efL_e^{-1}(e)+L_egL_e^{-1}(e)=fL_ef^{-1}L_e^{-1}L_egL_e^{-1}(e)=L_{f(e)}g(e)=a*b.\] 

For every block $[a]$ of $\alpha$, consider $h_{[a]}\in\dis(E)$ such that $h_{[a]}$ maps the block $[a]$ into the block $[e]$ (such mappings exist due to connectedness of $E/\alpha$). Now, consider a dynamical cocycle 
\[ \beta_{[a],[b]}(s,t)=h_{[a*b]}L_{h^{-1}_{[a]}(s)}h_{[b]}^{-1}(t). \] 
We will show that 
\[ \varphi:E\to E/\alpha\times_\beta A,\qquad a\mapsto([a],h_{[a]}(a)) \]
is an isomorphism. It is bijective, because the mappings $h_{[a]}$ are bijective on every block (cf. the proof of Proposition \ref{p:uniform}). For $a,b\in E$, we have 
\begin{align*}
\varphi(a)*\varphi(b)=([a],h_{[a]}(a))*([b],h_{[b]}(b)) &= ([a*b],\beta_{[a],[b]}(h_{[a]}(a),h_{[b]}(b))) \\
&=([a*b],h_{[a*b]}L_{h^{-1}_{[a]}(h_{[a]}(a))}h_{[b]}^{-1}(h_{[b]}(b))) \\
&=([a*b],h_{[a*b]}L_a(b))=\varphi(a*b).
\end{align*}

Now, set 
\[ \theta_{[a],[b]}=\beta_{[a],[b]}(e,e) \quad\text{and}\quad \psi_{[a],[b]}(t)=-\beta_{[a],[b]}(e,e)+\beta_{[a],[b]}(e,t). \]
We prove that the mappings $\psi$ actually do not depend on $a,b$. For $t=f(e)$ with $f\in\dis(Q)$, expand 
\[ \beta_{[a],[b]}(e,t)=h_{[a*b]}L_{h^{-1}_{[a]}(e)}h_{[b]}^{-1}(f(e))=h_{[a*b]}h^{-1}_{[a]}L_e h_{[a]}h_{[b]}^{-1}fL_e^{-1}(e) \]
in order to obtain a mapping from $\dis(Q)$ acting on $e$, and calculate
\begin{align*}
\psi_{[a],[b]}(t)&=-\beta_{[a],[b]}(e,e)+\beta_{[a],[b]}(e,t) \\
&=\left(h_{[a*b]}h^{-1}_{[a]}L_e h_{[a]}h_{[b]}^{-1}L_e^{-1}\right)^{-1}\left(h_{[a*b]}h^{-1}_{[a]}L_e h_{[a]}h_{[b]}^{-1}fL_e^{-1}\right)(e) \\
&=L_efL_e^{-1}(e)=L_e(t).
\end{align*}
In particular, $\psi=\psi_{[a],[b]}=L_e$ is an automorphism of $A$ (as proved earlier).

Finally, we show that $\beta_{[a],[b]}(s,t)=(1-\psi)(s)+\psi(t)+\theta_{[a],[b]}$, thus completing the proof that $E$ is isomorphic to a central extension. Again, for $s=f(e)$ and $t=g(e)$ with $f,g\in N$, we calculate
\begin{align*}
(1-\psi)(s)+\psi(t)+\theta_{[a],[b]}
&= f(e) - L_efL_e^{-1}(e) + L_egL_e^{-1}(e) + h_{[a*b]}h^{-1}_{[a]}L_eh_{[a]}h_{[b]}^{-1}L_e^{-1}(e) \\
&= f(L_ef^{-1}L_e^{-1})(L_egL_e^{-1})(h_{[a*b]}h^{-1}_{[a]}L_eh_{[a]}h_{[b]}^{-1}L_e^{-1})(e) \\
&= L_{f(e)}\underbrace{g}_{\in N}\underbrace{L_e^{-1}h_{[a*b]}h^{-1}_{[a]}L_e}_{\in\dis(E)}\underbrace{h_{[a]}h_{[b]}^{-1}}_{\in\dis(E)}(e) \\
&= \underbrace{L_{f(e)}L_e^{-1}}_{\in\dis_\alpha}  \underbrace{h_{[a*b]}h^{-1}_{[a]}}_{\in\dis(E)} L_e h_{[a]}h_{[b]}^{-1}g(e) \\
&= h_{[a*b]}h^{-1}_{[a]} L_{f(e)}L_e^{-1} L_e h_{[a]}h_{[b]}^{-1}g(e) 
= h_{[a*b]}h^{-1}_{[a]} L_s h_{[a]}h_{[b]}^{-1}(t)= \beta_{[a],[b]}(s,t)
\end{align*}
(in the second step, note that all mappings are in $\dis(Q)$; later, we used centrality of $N$ and of $\dis_\alpha$).
\end{proof}

\begin{corollary}
Let $E$ be a quandle and $\alpha$ its central congruence such that $E/\alpha$ is connected and $\dis_\alpha$ acts transitively on every block of $\alpha$. Then $E$ is isomorphic to a central extension $E/\alpha\times_{\phi,\psi,\theta} A$.
\end{corollary}

\begin{proof}
Apply Proposition \ref{p:central_ext_rep} to $N=\dis_\alpha$. Indeed, $\dis_\alpha\in\N(E)$ is central in $\dis(Q)$ by Theorem \ref{t:abelian,central}, and $\alpha=\mathcal{O}_{\dis_\alpha}$ by Proposition \ref{p:dis_alpha2}(3) and the transitivity condition.
\end{proof}


\begin{remark}\label{r:ext_ua}
Proposition \ref{p:central_ext_rep} applies to any central congruence of any latin quandle.  This has been known for a long time in the setting of quasigroup theory.
Our definition of central extensions is a special case of the general definition from \cite[Section 7]{FM} for arbitrary algebraic structures. If $A$ is an algebraic structure in a congruence modular variety, then any central congruence of $A$ admits a representation by a central extension \cite[Proposition 7.1]{FM}. In particular, this applies to any quasigroup $(Q,*,\ld,/)$; however, the right division operation is essential, congruences must be central with respect to terms in all three operations. Therefore, the two results are equivalent only in the finite case. We also refer to a similar result for loops \cite[Theorem 4.2]{SV2}.
\end{remark}

There is no direct analogy of Proposition \ref{p:central_ext_rep} for abelian congruences, not even for finite latin quandles. 

\begin{example}\label{e:abelian_ext_rep}
The latin quandles {\tt SmallQuandle}(28,11) and (28,12) in the RIG library have an abelian non-central congruence which does not admit a representation by an abelian extension. In both cases, the blocks are affine quandles over the group $\Z_7$, and the factor is the four-element latin quandle $Q_4$, whose multiplication table is below:
\begin{center}
\begin{tabular}{|c c c c |}
\hline
 1 & 3 & 4 & 2\\
 4 & 2 & 1& 3 \\
 2 & 4& 3& 1 \\
 3 & 1& 2& 4  \\\hline
\end{tabular}
\end{center}

Consider a quandle $Q_4$-module over the group $\Z_7$, i.e., $\phi:Q_4^2\to\End{\Z_7}$, $\psi:Q_4^2\to\Aut{\Z_7}$ satisfying equations \eqref{eq:ext2}, \eqref{eq:ext3}, \eqref{eq:ext4}. 
Let $\beta$ be the dynamical cocycle given by $\phi,\psi$ and an arbitrary admissible $\theta$. Let $\gamma_a=\psi_{a/1,1}^{-1}=\psi_{1*a,1}^{-1}$ and 
\[ \tilde\beta_{a,b}(s,t)=\underbrace{\gamma_{a\ast b}\phi_{a,b}\gamma_{a}^{-1}}_{\nu_{a,b}}(s)+\underbrace{\gamma_{a\ast b}\psi_{a,b}\gamma_{b}^{-1}}_{\varepsilon_{a,b}}(t)+\gamma_{a\ast b}(\theta_{a,b}). \]
According to \cite[Definition 2.6]{AG}, $\beta$ and $\tilde\beta$ are cohomologous, hence the corresponding extensions are isomorphic.
We will show that $\nu,\varepsilon$ are constant mappings, and therefore cannot represent a non-central congruence.

Obviously, $\varepsilon_{a,1}=\varepsilon_{1,1}$ for every $a\in Q_4$. Using the cocycle conditions and abelianness of $\aut{\mathbb{Z}_7}$, it is straightforward to verify that $\varepsilon_{1,1}=\varepsilon_{a,a}=\varepsilon_{1,a}$, $\varepsilon_{1\ast a,1\ast b}=\varepsilon_{a,b}$ and that $\nu_{1\ast a,1\ast b}=\nu_{a,b}$ for every $a,b\in Q_4$. For example, setting $b=c$ in \eqref{eq:ext2}, we obtain $\varepsilon_{a*b,a*b}\varepsilon_{a,b}=\varepsilon_{a,b}\varepsilon_{b,b}$, cancel $\varepsilon_{a,b}$ thanks to commutativity and use connectedness of $Q_4$ to conclude that all diagonal entries are equal. The other cases are proved similarly.
So, accordingly,
\[
\nu=\begin{bmatrix}
1-\lambda&\nu_0&\nu_0&\nu_0\\
       \nu_1&1-\lambda& \nu_2& \nu_3\\
       \nu_1&\nu_3&1-\lambda&\nu_2\\
       \nu_1&\nu_2& \nu_3& 1-\lambda
\end{bmatrix},
\quad \quad
\varepsilon=\begin{bmatrix}
\lambda &\lambda&\lambda&\lambda\\
        \lambda&\lambda& k & l\\
        \lambda& l &\lambda& k \\
        \lambda& k & l& \lambda
\end{bmatrix},
\]
for some $k,l,\lambda\in \aut{\mathbb{Z}_7}$ and $\nu_i\in \mathrm{End}(\mathbb{Z}_7)$ for $i=0,\dots,3$.
We are left with seven parameters, so it becomes feasible to set a computer search over all options, checking the cocycle conditions for each choice. Over the group $\Z_7$, all solutions satisfy $k=l=\lambda$ and $\nu_i=1-\lambda$ for all $i$, hence both $\nu,\varepsilon$ are constant. 
\end{example}

\section{Applications}\label{s:8}

\subsection{Non-existence results}

As an application of the commutator theory, we will show a few non-existence results. As an ingredient, we will use a part of the classification of finite simple quandles \cite{AG,Joyce-simple}: a finite simple abelian quandle is affine of prime power size (this is essentially the contents of \cite[Theorem 3.9]{AG}).


\begin{theorem}\label{t:2^k}\
\begin{enumerate}
	\item There is no connected involutory quandle of size $2^k$, for any $k\geq1$. 
	\item There is no connected involutory rack of size $2^k$, for any $k>1$. 
\end{enumerate}
\end{theorem}

\begin{proof}
(1) Let $Q$ be a connected involutory quandle of size $2^k$. According to Proposition \ref{p:p-group}, $\dis(Q)$ is a 2-group, hence nilpotent, and thus $Q$ is nilpotent by Theorem \ref{t:nilpotent,solvable}. Therefore, it has a simple abelian factor $Q/\alpha$, and thanks to Proposition \ref{p:uniform}, it has size $2^l$, $l\leq k$. Finite simple abelian quandles are affine, and thus latin. But there is no latin involutory quandle of even order, because left translations in latin quandles have precisely one fixed point.

(2) Consider the the smallest congruence $\alpha$ such that the factor is a quandle. It is uniform by Proposition \ref{p:uniform}, hence $|Q/\alpha|=2^k$, which is impossible unless $k=0$. Hence $\alpha=1_Q$ and $Q$ must be a permutation rack. But the only connected involutory permutation rack has two elements.
\end{proof}

Theorem \ref{t:stein} was originally proved by Stein \cite[Theorem 9.9]{SteS} in 1950s using a topological argument: from a graph of the corresponding latin square, he constructed a triangulated polyhedron, and discussed the parity of its Euler characteristic. In \cite[Theorem 6.1]{Gal}, Galkin proved Stein's theorem using a shorter group-theoretical argument about the minimal representation. Our theory allows a direct inductive proof.

\begin{theorem}\label{t:stein}\cite{SteS}
There is no latin quandle of size $\equiv 2\pmod 4$.
\end{theorem}

\begin{proof}
Let $Q$ be the smallest latin quandle of size $\equiv 2\pmod 4$. If $Q$ was simple, then it was abelian thanks to Corollary \ref{Latin are solv}, hence affine of prime power order; but no prime power is $\equiv 2\pmod 4$ with the exception of $2^1$, but there is no latin quandle of order 2, contradiction. So $Q$ has a non-trivial congruence $\alpha$ which is uniform by Proposition \ref{p:uniform}. Let $m$ denote the size of its blocks and $n$ the size of its factor. Then $|Q|=m\cdot n\equiv 2\pmod 4$, hence either $m$ or $n$ is $\equiv 2\pmod 4$, and this contradicts that $Q$ was the smallest with this property.
\end{proof}

\subsection{Coloring knots and links}

Quandle coloring is a powerful invariant of knot (and link) equivalence, particularly from the computational perspective \cite{CESY,FLS}. Coloring by affine quandles is related to the Alexander invariant: the main result of \cite{Bae} states that a link is colorable by an affine quandle if and only if its Alexander polynomial does not vanish. We extend the theorem to solvable quandles. The following lemma is essentially \cite[Lemma 1]{FLS}.

\begin{lemma}\label{l:simple coloring}
Let $c$ be a non-trivial coloring of a link $L$ by a quandle $Q$, and assume that $\mathrm{Im}(c)$ generates $Q$. Then $L$ is colorable by every simple factor of $Q$.
\end{lemma}

\begin{proof}
Consider any simple factor $R=Q/\alpha$, and take the composition $c'=\pi\circ c$ where $\pi$ is the natural projection $Q\to R$. Then $c'$ is a coloring of $L$ by $R$. If $c'$ was trivial, then all colors used by $c$ were in one block, $B$, of $\alpha$. Since congruence blocks are subquandles and $Q$ is generated by $\mathrm{Im}(c)$, we have $B=Q$, hence $\alpha=1_Q$, which is a contradiction.
\end{proof}

\begin{theorem}\label{t:coloring}
Let $L$ be a link with trivial Alexander polynomial, and $Q$ be a finite connected solvable quandle. Then $L$ is not colorable by $Q$.
\end{theorem}

\begin{proof}
Let $c$ be a non-trivial coloring and consider the subquandle $S$ generated by $\mathrm{Im}(c)$. Proposition \ref{p:HSP} implies that $S$ is also solvable, and Lemma \ref{l:simple coloring} says that $L$ is colorable by every simple factor of $S$. However, all simple factors of a solvable quandle are abelian, hence affine. This contradicts \cite[Theorem 1.2]{Bae} which says that links with trivial Alexander polynomial admit no non-trivial coloring by an affine quandle.
\end{proof}

In particular, Corollary \ref{Latin are solv} implies that links with trivial Alexander polynomial are not colorable by any finite latin quandle. (It agrees with the data calculated in \cite{CESY}.)

\subsection{Bruck loops}
This subsection is written for loop theory specialists, so we omit explaining the details of the definitions and facts, that can be found in \cite{Bruck}.

Recall that, in loop theory, the classical notion of central nilpotence is equivalent to nilpotence in the sense of universal algebra, but the classical notion of solvability is strictly weaker than universal algebraic solvability \cite{SV1}. 
In 1960's, Glauberman proved that Bruck loops of odd order are solvable in the weak sense \cite[Theorem 14]{G2}, that Bruck loops of prime power order $p^k$, $p\neq2$, are nilpotent \cite[Theorem 7]{G1}, and with some effort, one can deduce from the results of \cite{G1,G2} the converse, that nilpotent Bruck loops of odd order are isomorphic to a direct product of loops of prime power order.

There is a strong link between the theory of latin quandles, and the theory of Bruck loops. There is a polynomial equivalence between the variety of latin quandles, and the variety of so called \emph{Belousov-Onoi modules}, which consist of a Belousov-Onoi loop and its automorphism (see \cite{BO}, or \cite[Section 5.1]{Sta-latin}). Polynomial equivalence preserves all properties defined by polynomial operations, such as congruences, the centralizing relation $C(\alpha,\beta;\delta)$, and subsequently the notions of abelianness, solvability, etc. (see \cite[Section 4.8]{Bergman} for details). Therefore, our Corollary \ref{Latin are solv} and Theorem \ref{t:nilpotent_prime_decomposition} imply that Belousov-Onoi loops are solvable in the stronger sense, and that a finite Belousov-Onoi loop is nilpotent if and only if it decomposes to a direct product of loops of prime power size. 

In the Belousov-Onoi correspondence, involutory latin quandles correspond to uniquely 2-divisible Bruck loops \cite[Theorem 5.9]{Sta-latin}. Equivalently, in the finite case, to Bruck loops of odd order. Therefore, our theory strengthens the Glauberman's solvability theorem, and provides an alternative and complete proof of the prime decomposition theorem.

\begin{corollary}\label{c:bruck}\
\begin{enumerate}
	\item Bruck loops of odd order are solvable (in the stronger sense of universal algebra).
	\item A Bruck loop of odd order is nilpotent if and only if it is isomorphic to a direct product of Bruck loops of prime power order.
\end{enumerate}
\end{corollary}

\begin{proof}
Apply the polynomial equivalence of \cite[Theorem 5.9]{Sta-latin} to Corollary \ref{Latin are solv} and Theorem \ref{t:nilpotent_prime_decomposition}.
\end{proof}

A curious reader might ask: the original Glauberman's proof of solvability involved his famous Z$^*$-theorem, one of the key steps in the classification of finite simple groups, where is it hidden in our proof of the stronger theorem? The answer is, in Stein's proof that latin quandles have solvable left multiplication groups, which uses the classification of finite simple groups.


\end{document}